\documentclass[10pt, reqno]{amsart}

\usepackage{amsmath, amsthm, amssymb}
\usepackage{url}
\usepackage[shortlabels]{enumitem}

\usepackage{amsfonts}
\usepackage{bm}

\usepackage{tikz}
\usetikzlibrary{arrows, cd}
\usepackage[all,cmtip]{xy}
\usepackage{mathtools}
\usepackage{adjustbox}

\usepackage{hyperref}

\newcommand{\PP}{\mathbf P}

\newcommand{\AAA}{\mathbf A}

\newcommand{\nn}{\mathbf n}
\newcommand{\rr}{\mathbf r}
\newcommand{\dd}{\mathbf d}

\def\cC{{\mathcal C}}

\def\cE{{\mathcal E}}
\def\cF{{\mathcal F}}

\def\cL{{\mathcal L}}

\def\cO{{\mathcal O}}

\def\PC{\operatorname{PC}}
\def\rank{\operatorname{rank}}
\def\ord{\operatorname{ord}}
\def\dim{\operatorname{dim}}

\newcommand{\VO}{\operatorname{VO}}
\newcommand{\RO}{\operatorname{RO}}

\newtheorem{theorem}{Theorem}

\newtheorem{lemma}[theorem]{Lemma}

\theoremstyle{definition}
\newtheorem{definition}[theorem]{Definition}
\newtheorem{example}[theorem]{Example}

\theoremstyle{remark}

\numberwithin{equation}{section}

\newcommand{\SG}{\operatorname{SG}}

\def\Euc{\operatorname{Euc}}
\def\CW{\operatorname{CW}}
\def\GCD{\operatorname{GCD}}
\def\frack{\operatorname{frac}}

\begin{document}

\title[The structural invariants of Goursat distributions]{The structural invariants \\ of Goursat distributions}

\author[S.\ J.\ Colley]{Susan Jane Colley}
\address{Department of Mathematics,
Oberlin College, Oberlin, Ohio 44074, USA}
\email{scolley@oberlin.edu}

\author[G.\ Kennedy]{Gary Kennedy}
\address{Ohio State University at Mansfield, 1760 University Drive,
Mansfield, Ohio 44906, USA}
\email{kennedy@math.ohio-state.edu}

\author[C.\ Shanbrom]{Corey Shanbrom}
\address{California State University, Sacramento, 6000 J St., Sacramento, CA 95819, USA}
\email{corey.shanbrom@csus.edu}


\begin{abstract}
This is the first of a pair of papers devoted to the local invariants of Goursat distributions.
The study of these distributions naturally
leads to a tower of spaces over an arbitrary surface,
called the monster tower,
and thence to connections with the topic
of singularities of curves on surfaces.
Here we study those invariants of Goursat distributions
akin to those of curves on surfaces,
which we call structural invariants.
In the subsequent paper we will relate these structural
invariants to the small growth invariants.
\end{abstract}



\date{\today}
\maketitle


\section{Introduction} \label{intro}
This is the first of a pair of papers devoted to the local invariants of Goursat distributions.
A \emph{distribution} $D$ on a smooth manifold $M$ is a subbundle of the tangent bundle $TM$.
It is called \emph{Goursat}
if the \emph{Lie square sequence}
\[ 
D=D_{1} \subset D_{2} \subset D_{3} \subset \cdots 
\]
(as defined in Section~\ref{Goursatdef})
is a sequence of vector bundles for which
$$
\rank D_{i+1} =1+\rank D_i
$$
until one reaches the full tangent bundle.
As realized by Montgomery and Zhitomirski \cite{MR1841129},
the study of Goursat distributions naturally
leads to a tower of spaces over an arbitrary surface,
called the monster tower,
and thence to connections with the topic
of singularities of curves on surfaces.
In the local study of such curves and in the local study of
Goursat distributions, an important role
is played by invariants:
for curves on surfaces there are many
well-studied invariants,
such as those listed on page 85 of \cite{MR2107253},
whereas for Goursat distributions
we have notions related to the
analysis of the small growth sequence,
which we define in Section \ref{sgs}.
(We spell out precisely what we mean by ``invariant''
in Sections~\ref{Goursatdef} and \ref{fcg}.)

The aims of this pair of papers, taken together, are
\begin{enumerate}
\item \label{cli}
To give a systematic
account of those invariants of Goursat distributions
akin to those of curves on surfaces, including the
Puiseux characteristic ---
we call them \emph{structural invariants};
\item \label{sgi}
To explain how they lead to the standard
invariants of the small growth sequence:
the small growth vector,
Jean's beta vector (\cite{MR1411581}), and its derived vector;
\item \label{recur}
To present effective recursive methods for calculation.
\end{enumerate}

This paper is devoted to aim \ref{cli}
and the subsequent paper to aim \ref{sgi}\,;
our third aim is addressed throughout both.
Our analysis uses the monster tower,
which ties together the notions of prolonging a Goursat distribution
and lifting a curve.
Our starting point is the RVT codes,
which provide a way of assigning
a code word to each germ of a Goursat distribution
and to each irreducible curve germ.
Furthermore, our analysis is recursive.
In fact, there are two compatible notions of recursion in play,
which we call front-end
and back-end recursions.

Since we use the monster tower,
our account of invariants of curves on surfaces
is based on the idea of Nash modifications.
Specialists in singularity theory tend to be more familiar
with the theory of point blowups.
(Intuitively, Nash modification
uses tangent lines, whereas ordinary blowing-up
uses secant lines.)
Most of our basic definitions and all of our structural invariants
can be interpreted in the alternative framework
of embedded blowups;
Section~\ref{blowup} briefly explores this approach.
We have tried to give a relatively self-contained
exposition, and thus we sometimes offer
our own versions of proofs already found
elsewhere, especially in  \cite{MR1841129}.
We do this, in part, so as to provide convenient references
for our subsequent paper.
Our other chief references are  \cite{MR2599043},
\cite{MR2172057}, and \cite{MR2107253}. 

Figure \ref{invdiagram}
shows the invariants we will analyze in the two papers;
we will also explain the relationships indicated by the arrows.
Figure \ref{exdiagram} presents an example, using the
same layout.

%
%
\begin{figure}[!htb] \label{diaginv}

\[   \begin{tikzcd}[sep=0pt]
\fbox{RVT code word} \arrow[rr, leftrightarrow]
\arrow[dd] \arrow[dr, leftrightarrow] &   &
\fbox{Puiseux characteristic} 
\arrow[dl, leftrightarrow] \arrow[dd] \\  & 
\fbox{$
\substack{\text{\normalsize{Multiplicity sequence}} \\[3pt] \text{\normalsize{Vertical orders vector}}}
$}
  &    \\
\fbox{$
\substack{\text{\normalsize{Goursat} }\\[2pt] \text{\normalsize{code word}}} 
$}
\arrow[rr, leftrightarrow]  \arrow[dr, leftrightarrow] &   & 
\fbox{$
\substack{\text{\normalsize{Restricted}} \\[2pt] \text{\normalsize{Puiseux characteristic}}} 
$}
\arrow[dl, leftrightarrow] \\   &   
\fbox{$
\substack{\text{\normalsize{Small growth vector}} \\[3pt] \text{\normalsize{Jean's beta vector}} \\[3pt] \text{\normalsize{Derived vector}} \\[3pt] \text{\normalsize{Second derived vector}} \\[3pt] \text{\normalsize{Restricted vertical orders vector}}} 
$}
\arrow[uu, leftarrow, crossing over]  &   
\end{tikzcd} \]
\caption{The invariants listed in the top three boxes
are invariants of 
points on monster spaces
(as defined in Section~\ref{MT})
and of
focal curve germs 
(as defined in Section~\ref{fcg}),
akin to invariants of curves on surfaces;
the others are invariants of germs of Goursat distributions (as defined in Section~\ref{Goursatdef}).
The invariants listed in the bottom box are invariants obtained by
considering the small growth sequence
(as defined in Section~\ref{sgs}).}
\label{invdiagram}
\end{figure}
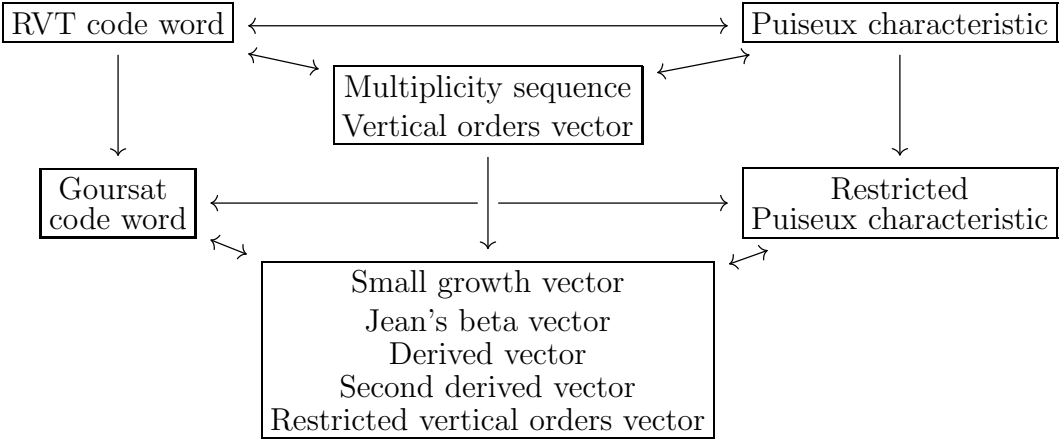

%
%
\begin{figure}[!t]
\begin{center}
\includegraphics[scale=0.85]{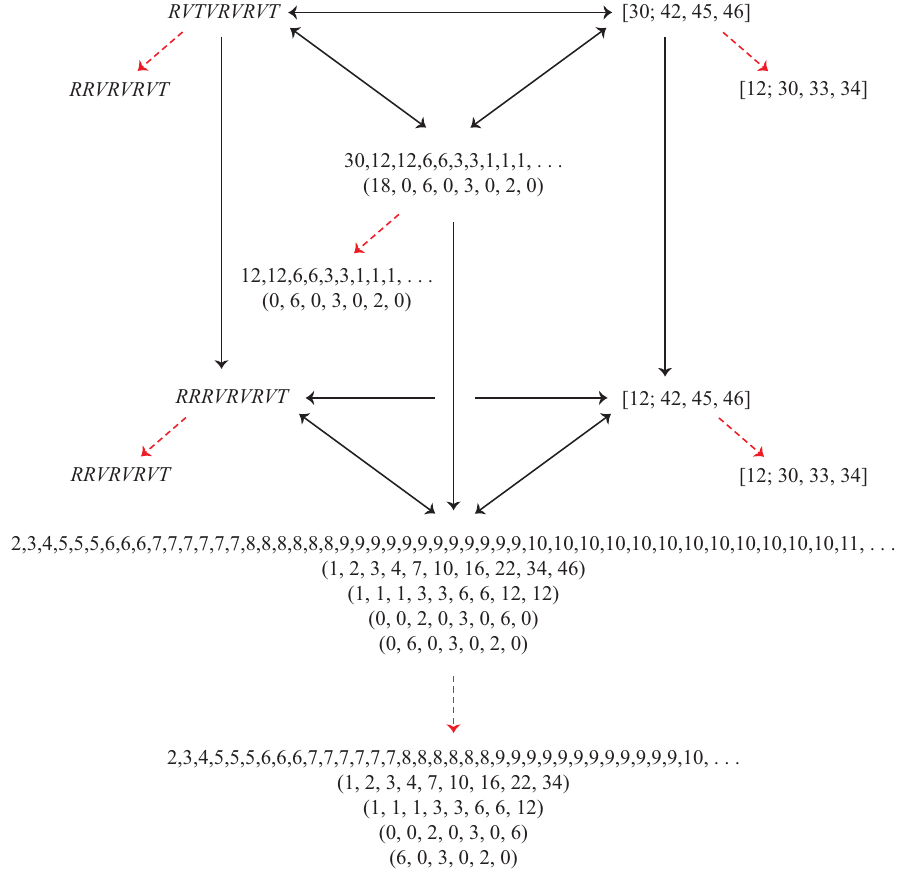}
\end{center}
\caption{An example of the invariants of Figure~\ref{diaginv}.
The dashed arrows indicate compatible front-end recursions.}
\label{exdiagram}
\end{figure}


Section~\ref{Goursatdef} recalls the basic definitions of Goursat
distributions, and Section~\ref{RVTGoursat}
explains how their basic structural theory leads to 
code words in the alphabet $R$, $V$, $T$.
Section~\ref{PMC} reviews the notion of prolongation,
while clarifying a ubitiquous construction
in differential geometry which seems never to have been
baptized; we coin the term \emph{extension}.
Section~\ref{MT} introduces the monster tower,
while Section~\ref{universality}, following \cite{MR1841129},
explicates how it is universal for Goursat distributions.
Section~\ref{fcg} explains the terminology of
focal curve germs.
Section~\ref{monstercoordinates}
explains how one naturally introduces
coordinates on the spaces of the monster tower.
Section~\ref{strata} gives
the basic stratification theory of divisors at infinity and their prolongations,
while Section~\ref{redux} uses this theory to
enlarge the scope of the theory of code words.
The brief Section~\ref{blowup}
is something of a digression:
it explores how several of the notions of the prior sections
can be interpreted using embedded resolution
via point blowups.
With Section~\ref{PC},
we begin to look at the structural invariants;
this section is devoted to the Puiseux characteristic,
while the subsequent 
Sections~\ref{multseq} and \ref{vert}
handle the multiplicity sequence
and the \emph{vertical orders vector} (again a new coinage).
The final Section~\ref{sgs}
is a brief introduction to the small growth sequence,
whose invariants will be the subject of our subsequent paper.

We thank Richard Montgomery, Piotr Mormul, Lee McEwan, Justin Lake, Tom Ivey,
and Fran Presas for valuable discussions.

\section{Goursat distributions} \label{Goursatdef}

We begin with a smooth manifold $M$ of dimension $m \ge 2$.
Let $D$ be a distribution on $M$, i.e., 
a subbundle of its tangent bundle $TM$.
Let $\cE$ be its sheaf of sections,
which is a subsheaf of the sheaf $\Theta_M$
of sections of $TM$.
In other words, let $\Theta_M$
be the sheaf of vector fields on $M$,
and let $\cE$ be the subsheaf of
vector fields
tangent to $D$.

The \emph{Lie square} of $\cE$ is
\[
\cE_{2}=[\cE,\cE],
\]
meaning the subsheaf of $\Theta_M$ whose sections are generated by
Lie brackets of sections of $\cE$ and the sections of $\cE$ itself.
Note that in general the rank of $\cE_{2}$
may vary from point to point.
If, however, $\cE_{2}$
is the sheaf of sections of a distribution $D_{2}$,
then this rank is constant.
Beginning with $\cE_{1}=\cE$,
recursively we define $\cE_{i+1}$ to be the Lie square of $\cE_{i}$,
and we call 
\begin{equation} \label{Liesquaresheaves}
\cE_{1} \subseteq \cE_{2} \subseteq \cE_{3} \subseteq \cdots 
\end{equation}
the \emph{Lie square sequence}.

Let $d$ be the rank of $D$.
We say that $D$ is a \emph{Goursat distribution}
if each sheaf $\cE_{i}$ is the sheaf of sections
of a distribution $D_{i}$
and if
$$
\rank D_{i+1} = 1 + \rank D_i
$$
for $i=1,\dots,m-d$.
In particular $D_{m-d+1}$ is the tangent bundle $TM$.
For a Goursat distribution, the sequence 
\[ 
D=D_{1} \subset D_{2} \subset D_{3} \subset \cdots \subset D_{m-d} \subset D_{m-d+1} = TM
\]
is also called the Lie square sequence.
Since a distribution of rank 1 is always integrable,
a Goursat distribution necessarily has rank at least 2.

In the literature on Goursat distributions, it is generally
assumed that the corank $m-d$
is at least 2. We find it more convenient to allow
these trivial cases:
\begin{itemize}
\item
$D=TM$,
\item
$D$ has corank 1, and its Lie square is $TM$,
\end{itemize}
and call all other cases \emph{nontrivial}.

Since we are interested in local issues,
we will work with germs of distributions.
A \emph{Goursat germ}
consists of a Goursat distribution $D$ on a manifold $M$ and a point $p \in M$;
if necessary for our constructions we will replace $M$ by a  neighborhood of $p$.
We say that the germ is \emph{located at} $p$.
Suppose that $(M,D)$ and $(M',D')$ are Goursat germs.
A local diffeomorphism from $M$ to $M'$ taking
$p$ to $p'$ is said to be an \emph{equivalence
of Goursat germs}
if its derivative takes $D$ to $D'$.
An \emph{invariant of Goursat germs} is a function
on the set of equivalence classes. Note that passage to a smaller
neighborhood does not change an invariant.
See Section~\ref{fcg} for further remarks about
Goursat invariants and their relation to focal
curve germ invariants (as defined therein).

\section{Code words of Goursat germs} \label{RVTGoursat}

Fix a Goursat germ $D$ at the point $p\in M$.
Let $d \ge 2$ be the rank of $D$
and let $m \ge d$ be the dimension of $M$.
We give here a self-contained account of how one defines
the \emph{Goursat code word} of $D$ at $p$;
this is
a word of length $m-d$ in the symbols $R$, $V$, and $T$,
which stand for \textit{regular}, \textit{vertical}, and \textit{tangent}.
(In the literature on Goursat distributions, these may be called
RVT code words, but we prefer to reserve that usage for the
slightly more general code words defined in Section~\ref{redux}.)
Our account draws upon four sources:
Section 9.1 of \cite{MR704040}
lays out the basic structural possibilities
for Goursat germs;
\cite{MR2172057} uses these possibilities to attach a code word
(using $G$ rather than $R$ and $S$ rather than $V$),
invoking the Sandwich Diagram from \cite{MR1841129};
the monograph
\cite{MR2599043} introduces the now-standard notation.


For a distribution $D$, let $\cL(D)$ denote its \textit{Cauchy characteristic}.
This is the sheaf of vector fields $v$ in $D$ that preserve $D$, meaning that  $[v, D] \subset D$.  
If $\cL(D)$ is of constant rank, then it is a subdistribution of $D$,
and we mildly abuse notation by writing $\cL(D)$ for both
the distribution and its sheaf of sections.

\begin{lemma}[Sandwich Lemma of \cite{MR1841129}] 
For a nontrivial Goursat distribution $D$,
its Cauchy characteristic has constant rank
and is therefore the sheaf of sections of a subdistribution.
We have
$$\cL(D) \subset \cL(D_2) \subset D$$
and both inclusions are of corank one.
\end{lemma}

\begin{proof}
We begin with some remarks on a skew-symmetric bilinear form
$\omega$
on a finite-dimensional vector space $V$. We say that
a subspace $W$ is \emph{isotropic} if the restriction
of $\omega$ to $W$ is trivial. (Surprisingly, this is
not fully standard terminology; it is found, however,
in \cite[Definition 1.2.4]{MR3930627}.) 
By the classification theory of such forms,
as found in \cite[Section 10.3]{MR0276251}
or \cite[Section 1.1]{MR1853077}, one infers that if
there is an isotropic subspace $W$ of codimension one
and if $\omega$ is nontrivial, then the kernel
of $\omega$ has codimension one in $W$.

Consider the map
\[
D_2 \times D_2 \to D_3/D_2
\]
obtained by composing the Lie bracket with the quotient map.
Choose a local basis of sections for $D_2$
and restrict this map to the fibers
at a point $p$:
\begin{equation}\label{ss1}
D_2 (p) \times D_2 (p) \to D_3/D_2 (p)
\end{equation}
Since $D_3/D_2 (p)$ is one-dimensional,
this is a skew-symmetric bilinear form.
The subspace $D(p)$ is isotropic and
of codimension one inside $D_2(p)$;
thus the kernel $\cL(D_2)(p)$ of (\ref{ss1}) has
codimension one inside $D(p)$.

Similarly consider the skew-symmetric bilinear form
\begin{equation}\label{ss2}
D (p) \times D (p) \to D_2/D (p)
\end{equation}
and observe that the subspace $\cL(D_2)(p)$
is isotropic and of codimension one inside $D(p)$;
thus the kernel $\cL(D)(p)$ of (\ref{ss2}) has
codimension one inside $\cL(D_2)(p)$.
\end{proof}

The Sandwich Lemma can be applied to every nontrivial member of the Lie square sequence
of a Goursat distribution to yield the following diagram, in which each inclusion has corank one.

\begin{equation}\label{sandwich}
\begin{array}{ccccccccccccc} 
 &  & D=D_1 & \!\!\!\!\subset\!\!\!\! & D_2 & \!\!\!\!\subset\!\!\!\! & \!\!\cdots\!\! & \!\!\!\!\subset\!\!\!\! & D_{m-d-1} & \!\!\!\!\subset\!\!\!\! & D_{m-d} & \!\!\!\!\subset\!\!\!\! & TM \\[10pt]
 &  & \cup & & \cup &  &  &  & \cup &  &  &  \\[10pt]
\cL(D_1) & \!\!\!\!\subset\!\!\!\! &  \cL(D_2) & \!\!\!\!\subset\!\!\!\! & \cL(D_3) &\!\!\!\!\subset\!\!\!\! & \!\!\cdots\!\! & \!\!\!\!\subset\!\!\!\! & \cL(D_{m-d}) &\ \
\end{array}
\end{equation}
By the Jacobi identity, each $\cL(D_i)$ is involutive, so the Frobenius theorem shows that it induces a foliation, called a \textit{characteristic foliation}.

To define the Goursat code word, we associate 
a symbol with each square 
\begin{equation}\label{littlesandwich}
\begin{array}{cccc} 
D_i & \!\!\!\!\subset\!\!\!\! & D_{i+1}  \\[10pt]
\ \ \cup & \ \ & \cup  \ \ \\[10pt]
\cL(D_{i+1}) & \!\!\!\!\subset\!\!\!\! & \cL(D_{i+2}) &\ \
\end{array}
\end{equation}
of diagram~(\ref{sandwich})
and then read off the word from right to left.
We also associate the symbol $R$ with each of the 
two incomplete squares at the right end, so that 
each Goursat code word begins with $RR$.
Thus the symbol associated with the square of (\ref{littlesandwich})
will be the symbol of the code word in position $m-d-i+1$.

To begin, looking at the fibers of our sheaves at $p$, we remark that
we know two natural fillings $W$ for the sandwich
$$ \cL(D_{i+1})(p) \subset W \subset D_{i+1}(p),$$
since we could use either the hyperplane $W=D_{i}(p)$ or 
the hyperplane $W=\cL(D_{i+2})(p)$.
We will be particularly interested in the case where the two fillings coincide.

\begin{definition}\label{def: singular}   If $D_{i}(p)=\cL(D_{i+2})(p)$ for some $i\in\{1, 2, \dots, m-d-2\}$, then we say that $D$ is \emph{singular} at $p$.
\end{definition}

If $D$ is not singular at $p$, then we associate to it the code word of length $m-d$ consisting entirely of the letter $R$.  
In the singular case, we find it convenient to introduce coordinates;
these coordinates were employed in \cite{MR704040},
and later authors call them \emph{Kumpera--Ruiz coordinates}.

\begin{lemma} \label{sandwichequations}
Suppose that $D$ is a nontrivial Goursat germ on $M$ at $p$.
Consider these sheaves:
\begin{equation*}
\begin{array}{ccccc} 
&& D & \!\!\!\!\subset\!\!\!\! & \phantom{D}D_{2}  \\[10pt]
&&  \cup & \ \ &   \ \ \\[10pt]
\cL(D) & \subset & \cL(D_{2}) &  &\ \
\end{array}
\end{equation*}
\begin{enumerate}
\item \label{ld2equations}
There is an ordered pair of coordinate functions $x$ and $y$, part of a
system of local coordinates,
for which the sheaf $\cL(D_{2})$ is defined
inside $D_2$ 
by $dx=dy=0$.
\item \label{p1}
For any such pair of coordinates, either
there is
a third coordinate $y'$
and a constant $C$ so that
$D$ is defined inside $D_2$
by $dy=(C+y')dx$
(we call this the \emph{ordinary} situation)
or there is a third coordinate $x'$
so that $D$ is defined inside $D_2$
by $dx=x'dy$
(we call this the \emph{inverted} situation).
\item \label{ldequations}
In the ordinary situation, the sheaf $\cL(D)$
is defined inside $D_2$ by $dx=dy=dy'=0$,
and in the inverted situation by $dx=dy=dx'=0$.
\end{enumerate}
\end{lemma}

Observe that if we reverse the order of $x$ and $y$,
the inverted situation becomes the ordinary situation with constant 0.
In the ordinary situation, the
following diagram shows the equations described in the lemma.
These are equations defining the sheaves inside $D_2$;
thus of course the box at top right is empty.

\begin{equation} \label{eqboxes}
\begin{array}{cccccc} 
& & \boxed{dy= (C+y')dx}  & \!\!\!\!\subset\!\!\!\! & \boxed{\phantom{XXX}} \\[10pt]
 & \ \ & \cup  \ \ \\
\boxed{\begin{array}{c} dx = 0 \\ dy= 0 \\ dy' = 0 \end{array}} & \!\!\!\!\subset\!\!\!\! & \boxed{\begin{array}{c} dx = 0 \\ dy= 0 \end{array}} &\ \
\end{array}
\end{equation}

\begin{proof}
As we have remarked, the distribution $\cL(D_{2})$
is involutive and thus comes from a foliation.
Since $\cL(D_{2})$ is of corank 2 in  $D_{2}$,
there must be 
local coordinates $x$ and $y$ for which
$\cL(D_{2})$ is defined inside $D_{2}$ by the vanishing
of $dx$ and $dy$.
Since $D$ is of corank one inside $D_{2}$,
at every point we have a single dependence between $dx$ and $dy$:
we must have $dy=(C+y')dx$ in some
neighborhood of $p$ for some coordinate function $y'$,
or $dx=x'dy$ for some coordinate function $x'$.
(The possibilities form a projective line, and the last possibility
accounts for the point at infinity.)

One finds that the differential ideal of $\mathcal L(D)$ is generated by $dx$, $dy$, and $dy'$ in the ordinary situation, and by $dx$, $dy$, and $dx'$ in the inverted situation.
This can be computed directly, but also appears in \cite[p. 463]{MR1841129}, and follows from the Retraction Theorem for exterior differential systems 
(\cite[Proposition 6.1.17]{MR2003610} or
\cite[Theorem 1.3]{MR1083148} or
\cite[Theorem 5.4]{MR891190}).
 \end{proof}

We return to the definition of the Goursat code word,
by probing further the case where 
$D$ is singular at $p$.
If $D_{i}(p)=\cL(D_{i+2})(p)$ for some $i\in\{1, 2, \dots, m-d-2\}$, then we associate the letter $V$ with the square of (\ref{littlesandwich}).
It will be convenient to introduce the notation
\begin{equation}\label{eq: vertical locus}
N(V)_i = \{p \in M : D_{i}(p)=\cL(D_{i+2})(p)\}.
\end{equation}

We claim that there are local coordinates so that the
sheaves of (\ref{littlesandwich})
are defined inside $D_{i+2}$
by the corresponding equations shown here:
\begin{equation*}
\begin{array}{cccc} 
\boxed{\begin{array}{c} dy= (C+y')dx \\ dx=x'dy' \end{array}} & \!\!\!\!\subset\!\!\!\! & \boxed{dy= (C+y')dx}  \\[10pt]
\ \ \cup & \ \ & \cup  \ \ \\
\boxed{\begin{array}{c} dx = 0 \\ dy= 0 \\ dy' = 0 \end{array}} & \!\!\!\!\subset\!\!\!\! & \boxed{\begin{array}{c} dx = 0 \\ dy= 0 \end{array}} &\ \
\end{array}
\end{equation*}
Indeed, since we are not concerned with the ordering of $x$ and $y$,
we can assume that in the square that would appear
just to the right we are in the ordinary situation.
Thus, except for the box at top left, our claim is immediate from 
Lemma~\ref{sandwichequations} applied to $D_{i+2}$.
If we now consider the equations of these sheaves
inside $D_{i+1}$, then the equations at the bottom left
reduce to $dx=dy'=0$. Thus $x$ and $y'$ are appropriate
Kumpera--Ruiz coordinates
for which we can apply part~\ref{p1} of Lemma~\ref{sandwichequations}.
Among the possibilities given, the only one for which
$p\in N(V)_i$
is the inverted situation: $dx=x'dy'$.
Furthermore we see that locally the locus $N(V)_i$ is given by $x'=0$, a nonsingular hypersurface.

We now consider the square of (\ref{littlesandwich})
together with the square to its left:
\begin{equation}\label{twosquares}
\begin{array}{cccccc} 
D_{i-1} & \!\!\!\!\subset\!\!\!\! & D_i & \!\!\!\!\subset\!\!\!\! & D_{i+1}  \\[10pt]
\ \ \cup & \ \ & \cup & \ \ & \cup \ \ \\[10pt]
\cL(D_{i}) & \!\!\!\!\subset\!\!\!\! & \cL(D_{i+1}) & \!\!\!\!\subset\!\!\!\! & \cL(D_{i+2}) &\ \
\end{array}
\end{equation}
Restricting our attention to $N(V)_i$, we consider the intersection 
of its tangent sheaf $\Theta_{N(V)_i}$ with $D_i$. This is of corank one inside $D_i$,
and is cut out by the additional equation $dx'=0$.
If the fibers at $p$ of $D_{i-1}$ and 
$\Theta_{N(V)_i} \cap D_i$ are the same, then we assign the letter $T$
to the square on the left.  
Again it will be convenient to introduce notation for the locus of points for  where we have assigned
$V$ to the right square and $T$ to the left square: we let
\begin{equation}\label{eq: VT locus}
N(VT)_i = \{p \in M : D_{i}(p)=\cL(D_{i+2})(p) \ \text{and} \ D_{i-1}(p)= 
(\Theta_{N(V)_i} \cap D_i)(p)\}.
\end{equation}

Assuming that we have assigned $T$ to 
the left square of diagram ~(\ref{twosquares}),
we claim that the equations defining the sheaves
of this diagram
are as follows:
\begin{equation*}
\begin{array}{cccccc} 
\boxed{\begin{array}{c} dy= (C+y')dx \\ dx=x'dy' \\ dx' = x''dy' \end{array}} & \!\!\!\!\subset\!\!\!\! & \boxed{\begin{array}{c} dy= (C+y')dx \\ dx=x'dy' \end{array}}& \!\!\!\!\subset\!\!\!\! & \boxed{dy= (C+y')dx}  \\[10pt]
\ \ \cup & \ \ & \cup & \ \ & \cup \ \ \\
\boxed{\begin{array}{c} dx = 0 \\ dy= 0 \\ dy' = 0 \\ dx' =0 \end{array}} & \!\!\!\!\subset\!\!\!\! & \boxed{\begin{array}{c} dx = 0 \\ dy= 0 \\ dy' = 0 \end{array}} & \!\!\!\!\subset\!\!\!\! & \boxed{\begin{array}{c} dx = 0 \\ dy= 0 \end{array}} &\ \
\end{array}
\end{equation*}
To obtain the equations in the bottom left box, 
we apply part~\ref{ldequations}
 of Lemma~\ref{sandwichequations}, using $D=D_{i+1}$.
 This tells us that $\cL(D_i)$ is defined inside $D_{i+1}$
 by $dx=dy'=dx'=0$; adjoining the equation $dy=(C+y')dx=0$
 of $D_{i+1}$ inside $D_{i+2}$ and simplifying,
 we obtain the four indicated equations. To obtain the equations
 in the top left box, we apply part~\ref{p1} of Lemma~\ref{sandwichequations} using $D=D_i$.
Since we have assigned $T$ to the left square,
we are in the ordinary situation with constant $C=0$;
thus $dx'=x''dy'$ should be in our system of equations,
together with the two equations cutting out $D_{i-1}$ inside $D_{i+1}$.
We remark that the letters $V$ and $T$ are mutually exclusive,
since the former requires the inverted situation and the latter the
ordinary situation with $C=0$.

Furthermore we see that the locus $N(VT)_i$ 
is defined
inside $N(V)_i$ by the additional equation $x''=0$.
Thus it is a nonsingular submanifold of $M$ of codimension two.
Restricting our attention to $N(VT)_i$, we consider the intersection 
$\Theta_{N(VT)_i} \cap D_{i-1}$, cut out inside
$D_{i-1}$ by the single equation $dx''=0$.
If the fibers at $p$ of $D_{i-2}$ and 
$\Theta_{N(VT)_i} \cap D_{i-1}$ are the same, 
then we assign the letter $T$
to the next square to the left.

At this point it becomes clear that all arguments can be repeated:
for $\tau>1$ we recursively define 
\begin{equation}\label{eq: VTtau locus}
N(VT^{\tau})_i = N(VT^{\tau-1})_i \cap \{p \in M : \ D_{i-\tau}(p) =(\Theta_{N(VT^{\tau-1})_i} \cap D_{i-(\tau-1)})(p)\}.
\end{equation}
The locus $N(VT^{\tau})_i$ will be a nonsingular
submanifold in $M$ of codimension $\tau+1$,
with explicit equations $x'=x''=\cdots=x^{(\tau+1)}=0$;
and we will assign yet another $T$ if
\[
D_{i-(\tau+1)}(p) = (\Theta_{N(VT^{\tau})_i} \cap D_{i-\tau})(p).
\]

Having explained when to associate $V$ or $T$ with a square
of diagram~(\ref{sandwich}), to finish the definition of the Goursat code word
we declare that in all other circumstances we associate the symbol $R$.
Thus the possibilities for Goursat code words are circumscribed as follows:
\begin{itemize}
\item
The first two symbols are $RR$.
\item
The symbol $T$ may only be used immediately following a $V$ or $T$.
\end{itemize}

The arguments used to prove Lemma~\ref{sandwichequations}
also give an Existential Sandwich Lemma, as follows.

\begin{lemma} \label{esl}
Suppose that $D$ is a nontrivial Goursat distribution
of rank greater than 2.
Any distribution $E$ sandwiched between $\cL(D)$ and $D$
(i.e., of corank one in $D$ and having  $\cL(D)$ as a corank one
subbundle) is Goursat, with Lie square $E_2=D$
(i.e., $E$ is a ``Lie square root'' of $D$).
The possibilities for $E$ form a projective line,
and the corresponding code word symbol can be any one
of the possible symbols, subject to the constraints just mentioned.
 \end{lemma}

\section{Prolongation and extension} \label{PMC}

The work of Montgomery and Zhitomirskii
\cite{MR1841129,MR2599043}
shows that the study of Goursat distributions
naturally leads to the construction of a tower
of spaces.
Unbeknownst to them, the construction had previously been used
in algebraic geometry, but had never been connected to Goursat distributions.
(This earlier work includes three papers by the first two authors, including
\cite{MR1287696}.)

Once again we work with a smooth manifold $M$ carrying a distribution $D$
of rank $d$.
We now describe a \emph{general prolongation construction}
that creates a new manifold and new distribution.
Let $\widetilde{M}=\PP D$, the total space of the projectivization
of the bundle, and let $\pi : \widetilde{M} \to M$ be the projection.
A point $\widetilde{p}$ of $\widetilde{M}$ over $p \in M$
represents a line inside the fiber of $D$ at $p$,
and since $D$ is a subbundle of $TM$,
this is a \emph{tangent direction} to $M$ at $p$.
Let 
\[
d\pi : T{\widetilde{M}} \to \pi^*TM
\] 
denote the derivative map of $\pi$.
A tangent vector to $\widetilde{M}$ at $\widetilde{p}$
is said to be a \emph{focal vector}
if it is mapped by $d\pi$ to a tangent vector at $p$
in the direction represented by $\widetilde{p}$;
in particular a vector mapping to the zero vector
(called a \emph{vertical vector})
is considered to be a focal vector.
The subspace of focal vectors is called the \emph{focal space}.
The set of all focal vectors forms a subbundle
$\widetilde D$ of $T{\widetilde{M}}$,
called the \emph{prolongation} of $D$
or the \emph{focal bundle};
its rank is again $d$.
The vertical vectors form a subbundle, called
the \emph{vertical bundle} or the \emph{relative tangent bundle},
denoted $T({\widetilde{M}/M})$.
The terminology of focal vectors, focal bundles, focal spaces, etc.,
dates back to work of Semple \cite{MR0061406}.

In their Proposition 5.1,
Montgomery and Zhitomirskii \cite{MR1841129} establish the following
fundamental fact:
If $D$ is a Goursat distribution of rank 2, then
so is~$\widetilde D$.

To be more precise,
we need to clarify a construction in differential geometry
 which  seems to lack a standard terminology or notation.
Let
$\varphi : N \to M$
be a submersion
and let 
$\cF$ be a sheaf of vector fields on $M$.
The \emph{extension} of $\cF$ by $\varphi$,
denoted $\varphi^{\bullet}\cF$,
is the sheaf of all vector fields $v$ 
on open sets of $N$ satisfying the condition
\[
d\varphi(v(q)) \in \cF (\varphi(q))
\]
at each point $q$ of the open set. 
If $\cF$ is the sheaf of sections of a distribution $D$,
then $\varphi^{\bullet}\cF$ is the sheaf of sections of a distribution, 
which we denote by $\varphi^{\bullet}D$.
Note that our definition says $v \in \varphi^{\bullet} \cF$ if $ d\varphi \circ v$ is a section of the usual pullback bundle $\varphi^*\cF$.
Thus we are not describing the usual pullback of a bundle, nor is this the usual notion of an inverse image sheaf or a pullback sheaf in algebraic geometry, as described in  \cite{MR0463157}.
It is, however, consistent with the usage of 
\cite{MR1240644}
and  \cite{MR1841129}.
These authors use the potentially ambiguous notation $\varphi^*$;
hence our introduction of new notation and terminology.


Returning to the prolongation construction,
we observe that for a distribution $D$ on $M$, we have relations among
 vector bundles on $\widetilde{M}$ as indicated in the following diagram.
\begin{equation*}
\xymatrix @R=24pt {
T(\widetilde{M}/M)\: \ar@{=}[d] \ar@{^{(}->}[r] & \widetilde{D}_{} \ar@{_{(}->}[d] \ar@{>>}[r] & \cO_{D_{}}(-1) \ar@{_{(}->}[d]  \\
T(\widetilde{M}/M)\: \ar@{^{(}->}[r]  & \pi^{\bullet}D \ar@{>>}[r] & \pi^*D 
}
\end{equation*}
The four bundles on the left are distributions, i.e., subbundles of $T\widetilde{M}$.
The ranks of the bundles in the top row are (from left to right)
$d-1$, $d$, and 1; the ranks in the bottom row are 
$d-1$, $2d-1$, and $d$.  Here $ \cO_{D_{}}(-1) $ denotes the tautological line bundle.

Suppose that near $p \in M$
we have a set of equations cutting out $D$ inside $TM$.
This will be a set of linear equations in the differentials
of the local coordinates with coefficients in the local coordinates.
At a point $q \in \widetilde{M}$ over $p$ we can use a system of local coordinates including the local coordinates
pulled back from $M$;
then the very same set of equations cuts out the extension
$\pi^{\bullet}D$
 inside $T\widetilde{M}$.
To give equations for the prolongation $\widetilde{D}$, let us assume for simplicity that
$D$ has rank 2. Among our coordinate functions at $p$ we can find a pair
$x$ and $y$ such that their differentials $dx$ and $dy$ are independent linear functionals on the fiber of $D$ at $p$. Then $ \widetilde{D}$ is cut out
inside $\pi^{\bullet}D$ by one additional equation expressing the dependence
of $dx$ and $dy$; 
this will be either 
$dy=(C+y')dx$ (for a suitable local coordinate $y'$)
or $dx=x'dy$ (for a suitable local coordinate $x'$).

Returning to the situation 
of Montgomery and Zhitomirskii's Proposition 5.1,
what they assert is that if
\[ 
D \subset D_{2} \subset D_{3} \subset \cdots 
\]
is the Lie square sequence of a 
Goursat distribution $D$ of rank 2,
then
\[ 
\widetilde D \subset \pi^{\bullet}D \subset \pi^{\bullet}D_{2} \subset \pi^{\bullet}D_{3} \subset \cdots 
\]
is the Lie square sequence of $\widetilde{D}$, which is again Goursat of rank 2.
The assertion is clear from the general properties of extension, except
for the claim that $\widetilde D$ is Goursat.
To see this, one simply remarks that the 
relative tangent bundle $T({\widetilde{M}/M})$ is an involutive 
subbundle of $\pi^{\bullet}D$ of corank 2, and thus must be $\cL(D)$;
then Lemma~\ref{esl} shows that $\widetilde D$ is Goursat
with Lie square $\pi^{\bullet}D$. Note that
\begin{equation}\label{extendvsLiesquare}
\widetilde D_{2}=\pi^\bullet D
\end{equation}
and more generally that $\widetilde D_{i+1}=\pi^\bullet D_i $.

Since the general prolongation construction, when applied to a manifold $M$ carrying 
a distribution $D$, yields a new manifold
$\widetilde{M}$ carrying a new distribution $\widetilde{D}$,
we can iterate this construction to obtain
an infinite tower of manifolds and submersions
\begin{equation} \label{towermaps}
\cdots \to \widetilde{\widetilde{M}} \to \widetilde{M} \to M
\end{equation}
together with their associated focal bundles.

We consider a curve $C$ on $M$ which has a nontrivial smooth parameterization;
henceforth we will just say ``curve.''
Suppose that
$p$ is a nonsingular point on $C$.
If at $p$ the tangent vector of $C$ is contained in the fiber of $D$,
we say $C$ is \emph{tangent} to $D$ at $p$.
We say that $C$ is a \emph{focal curve} if it is tangent
to $D$ at each nonsingular point.
We can associate with $p$ the point
of $\widetilde{D}$ representing the tangent direction of $C$ at $p$.
Thus, away from singularities, we have a curve $\widetilde{C}$
in $\widetilde{M}$,
the \emph{lift} of $C$.
We also want to associate a point (or perhaps several points) of $\widetilde{M}$
with a singular point of $C$, and we do so by fiat:
we lift at all nonsingular points of $C$, and then take the closure.
Intuitively, we are associating to a singular point all the possible
limiting tangent directions.
We observe that the lift of a focal curve on $M$ gives us a focal curve on $\widetilde{M}$,
i.e., the lift of such a curve will be tangent to the focal bundle. Thus, if we like,
we can iterate the construction by further lifting.

\section{The monster tower} \label{MT}

If we begin by letting $D$ be the tangent bundle $TM$,
then the tower of (\ref{towermaps}) is called
the \emph{monster tower} (or \emph{Semple tower})
over the base manifold $M$.
In the tower
\begin{equation} \label{generalMT}
\cdots \to M(k) \xrightarrow{\pi_{k}} M(k-1) \xrightarrow{\pi_{k-1}} \cdots \to M(2) \xrightarrow{\pi_2} M(1) \xrightarrow{\pi_1} M
\end{equation}
we say that $M(k)$ is the \emph{monster space}
at \emph{level}~$k$. 
As above, we set $m = \dim(M) \geq 2$.
Each monster space $M(k)$ is the total space of a fiber bundle over $M(k-1)$, with fiber
a projective
space of dimension $m-1$.
The focal bundle on $M(k)$ is denoted $\Delta(k)$.
 
In general this has nothing to do with Goursat distributions.
If, however, we begin with a smooth surface $S$, then its tangent
bundle is a trivial Goursat distribution of rank 2; thus the
focal bundles are likewise rank 2 Goursat distributions.
In this case the fibers of the maps in (\ref{generalMT})
are projective lines. The Lie square sequence of the Goursat distribution
$\Delta(k)$ is
\[ 
\Delta(k) \subset \Delta(k)_{2} \subset \Delta(k)_{3} \subset \cdots \subset \Delta(k)_{k} \subset \Delta(k)_{k+1} = TS(k).
\]
Continuing to assume that the base manifold is a surface $S$,
we apply (\ref{extendvsLiesquare}) to $\pi_{k}$;
this tells us that
the Lie square of $\Delta(k)$ can be obtained by extension:
\begin{equation} \label{onestep}
 \Delta(k)_{2}=\pi^{\bullet}_{k}\, \Delta(k-1).
\end{equation}
More generally, the other bundles in its Lie square sequence are extensions from lower levels:
\begin{equation} \label{moresteps}
\Delta(k)_{j}=\pi^{\bullet}\Delta(k-j+1),
\end{equation}
where $\pi:S(k)\to S(k-j+1)$.

If we apply the lifting construction to the monster tower,
the lift of a focal curve on $S(k)$ --- a curve tangent to $\Delta(k)$ --- is a focal curve on $S(k+1)$,
and we can iterate the construction to
lift upward any desired number of levels in the tower.
If in particular we start with a curve $C$ on $S$
(automatically a focal curve)
we obtain curves $C(1)$ on $S(1)$, $C(2)$ on $S(2)$, etc.

As we have said, lifting associates with a singular point all the possible limiting tangent directions.
For example, if the curve $C$ has a node at $p$, then $C(1)$ will have two points
over $p$, and if $C$ has a cusp then it has a single point over $p$,
but this will be a nonsingular point on the threefold $C(1)$.
In the mathematical literature, $C(1)$ is also called the \emph{Nash modification}
of $C$.

\section{Universality of the monster tower} \label{universality}

As we previously noted, the monster spaces over a smooth surface $S$
naturally carry Goursat distributions of rank 2.
In \cite{MR1841129},
Montgomery and Zhitomirskii show that these spaces
are universal for Goursat distributions of rank 2.
More precisely, they show that each nontrivial Goursat germ
of rank 2
on a manifold of dimension $m=2+k$ (with $k\ge 2$) is equivalent
to the Goursat germ of $\Delta(k)$ at some point
of $S(k)$.
They prove this by a process which they call 
deprolongation,
which we explicate here, and explain how it leads to the universality.
(They also define a deprolongation process for Goursat germs
of higher rank, but we do not consider that process here.)

Consider
a nontrivial Goursat germ $D$ of rank 2 at the point $p\in M$.
Recall that the Cauchy characteristic $\cL(D_2)$ 
induces a characteristic foliation;
the leaves are curves.
Locally, one can contract these curves,
creating a smooth manifold $M/\cL(D_2)$,
the \emph{leaf space},
equipped with a submersion
$\lambda:M \to M/\cL(D_2)$.
Since each of the bundles in the Lie square sequence
of $D$ 
contains the bundle $\cL(D_2)$, we have a sequence
\begin{equation}\label{contractedseq}
\lambda_*(D/\cL(D_2)) \subset \lambda_*(D_{2}/\cL(D_2)) \subset \lambda_*(D_{3}/\cL(D_2)) \subset \cdots 
\end{equation}
of bundles on $M/\cL(D_2)$, with 
$\rank(\lambda_*(D_{j}/\cL(D_2))=\rank(D_j)-1=j$.
Each of these bundles extends to the corresponding
bundle in the Lie square sequence
of $D$:
\[
\lambda^{\bullet}\Bigl(\lambda_*(D_j/\cL(D_2))\Bigl) = D_j.
\]
We claim that if we omit the first bundle $\lambda_*(D/\cL(D_2))$
of  (\ref{contractedseq}) we obtain a Lie square sequence
\begin{equation}\label{shortcontractedseq}
\lambda_*(D_{2}/\cL(D_2)) \subset \lambda_*(D_{3}/\cL(D_2)) \subset \cdots.
\end{equation}
Indeed,
for $j\ge 2$ the Cauchy characteristic
$\cL(D_j)$ contains $\cL(D_2)$;
thus the Lie square map $D_j \times D_j \to D_{j+1}$
yields a surjection from
$D_j/\cL(D_2) \times D_j/\cL(D_2)$
to
$D_{j+1}/\cL(D_2)$.

Thus $\lambda_*(D_{2}/\cL(D_2))$
is a Goursat distribution of rank 2 on the leaf space.
We call it the \emph{deprolongation} of $D$.

Part~\ref{ld2equations} of Lemma~\ref{sandwichequations}
tells us that locally $\cL(D_2)$
is defined inside $D_2$ by the vanishing of $dx$ and $dy$,
where $x$ and $y$ are part of a system of local coordinates.
Thus these two functions are constant on the leaves,
and descend to functions on the leaf space.
Since the total space of $\cL(D_2)$ is of codimension 2
inside that of $D_2$, the differentials $dx$ and $dy$
remain independent when restricted to 
$\lambda_*(D_{2}/\cL(D_2))$.
Part~\ref{p1} of the same lemma
then tells us that locally each point of $p \in M$
can be interpreted as a tangent direction
to the point $\lambda(p)$
lying inside $\lambda_*(D_{2}/\cL(D_2))$,
and that $D$ is the focal bundle.
Thus if we apply the prolongation construction
to $\lambda_*(D_{2}/\cL(D_2))$,
we obtain
a projective line bundle over the leaf space
that contains the neighborhood on $M$ with which we started,
and the prolongation of 
$\lambda_*(D_{2}/\cL(D_2))$
is $D$.
In this sense, the operations of prolongation
and deprolongation are inverse processes.

Deprolongation can be applied repeatedly. If we begin
with a nontrivial Goursat germ $D$ of rank 2 and corank $k\ge 2$,
then we can deprolong it $k-1$ times, arriving at
a (trivial) Goursat germ of rank 2 and corank 1,
i.e., a distribution germ whose Lie square is the tangent bundle
of the base threefold. It is well known
that there is just one such distribution
(up to equivalence of germs of distributions),
the \emph{contact distribution}:
in appropriate local coordinates $x$, $y$, $y'$,
it is the distribution defined by $dy=y'dx$.
Since prolongation and deprolongation are inverse
processes, we conclude that
$D$ is found somewhere
in the prolongation tower over this distribution,
i.e., it is equivalent to the germ of the prolongation of
the contact distribution at some point
in the $(k-1)$st space in this tower.
As the contact distribution is equivalent to $\Delta(1)$ at any point of $S(1)$, this shows that $D$ appears somewhere in the monster tower.

One could actually stop this process one step earlier,
arriving at a Goursat germ whose rank and corank are both 2.
Again one knows that such a germ is unique:
it is the \emph{Engel distribution},
given in local coordinates $x$, $y$, $y'$, $y''$
by the vanishing of $dy-y'dx$ and $dy'-y''dx$.
See Sections 6.2.2 and 6.11 of \cite{MR1867362}.  The Engel distribution is equivalent to $\Delta(2)$ at any point of $S(2)$.

Alternatively, one can continue it by one additional step,
but this step is different, since it is not dictated
by the Cauchy characteristic. Given the germ
of the contact distribution, one can deprolong
in any direction tangent to this distribution, creating a surface
whose tangent bundle will prolong to the contact distribution.
We illustrate this by two examples.
\begin{example} \label{finaldeprolong1}
The coordinate names $x$, $y$, $y'$ naturally suggest
that one should contract the curves on which $x$ and $y$ are constant;
on the resulting surface with coordinates $x$ and $y$ one
interprets $y'$ as $dy/dx$.
\end{example}
\begin{example} \label{finaldeprolong2}
Alternatively, let $X=y'$, let $Y=y-y'x$, and let $Y'=-x$.
(To motivate this, think of $X$ and $Y$ as
``slope'' and ``intercept.'')
Since $dY = dy - y'dx - xdy'$ and $dy - y'dx$ vanishes on the contact
distribution, we see that in our new coordinates the contact distribution is defined
by the equation $dY=Y'dX$. Thus one can equally well
obtain a surface by contracting the curves on which $X$ and $Y$
are constant, and its tangent bundle will likewise prolong to the 
contact distribution.
\end{example}

Observe that, given a Goursat germ $D$ of rank 
2 and corank $k$, the process we have described
for finding a point $p_k$
on the monster space $S(k)$ over a surface $S$
is a recursive process. By repeated deprolongation
we first construct the surface $S$ and a point $p_0$;
then by the prolongation construction we build
$S(1)$ and a point $p_1$ lying over $p_0$, then $S(2)$ and $p_2$,
etc.

What about Goursat distributions of rank greater than two?
Given a nontrivial Goursat germ $D$ of rank $d\ge 3$,
Lemma~\ref{esl} can be applied $d-2$ times to obtain
the beginning of a Lie square sequence
\[ 
E \subset E_{2} \subset  \cdots \subset E_{d-1} = D,
\]
where $E$ is a rank 2 Goursat germ.
As we have just observed,
$E$ can be obtained by repeated prolongation
of the (threefold) contact distribution.
Letting $j$ and $k$ denote the coranks of $D$ and $E$ respectively, 
we observe that $k=j+d-2$;
thus $j+d-3$ steps of prolongation are required.
We conclude that any Goursat germ can be found within the monster tower over a surface.

\section{Focal curve germs} \label{fcg}

Throughout the remainder of the paper, we work in the monster tower
over a smooth surface $S$.
We will be studying local invariants, and thus
we are interested in a \emph{focal curve germ}
consisting of a point $p$ on some monster space $S(k)$,
together with a locally irreducible focal curve passing through $p$;
if necessary we will replace $S(k)$ by a  neighborhood of $p$.
We say that the germ is \emph{located at} $p$
or that $p$ is its \emph{location}.
Note in particular that every irreducible curve germ on $S$ is automatically a focal curve germ.

A focal curve germ on $S(k)$ may project to a single point of $S$.
Such a germ is said to be \emph{critical}.
We are using the terminology of Montgomery and Zhitomirskii \cite[Definition 2.16]{MR2599043}, but our definition differs slightly from theirs,
since they do not apply the terminology to a curve in $S(1)$.
This reflects a difference in viewpoint: they are studying
Goursat distributions, whereas we are equally focused
on curves as objects of intrinsic interest.
One can easily characterize a critical germ:
it is either the germ of a fiber of $S(k)$ over $S(k-1)$,
or the lift from some lower level of such a germ.
A tangent direction to a critical germ is also called \emph{critical}.

If the focal curve germ $C(k)$ located at $p \in S(k)$ is not critical, then
one may project it to the base surface $S$, obtaining a curve germ $C$
there, and then recover $C(k)$ by repeated lifting of $C$. 
(Note that
this justifies the notation $C(k)$.)

We say that $C(k)$ is \emph{regular} if it is nonsingular and if
the tangent direction at $p$ is not critical.
Starting with an arbitrary focal curve germ, we can lift it through the tower.
If, after a finite number of steps, we obtain a regular focal curve germ,
then we say the original germ is \emph{regularizable}.
Its \emph{regularization level} $r$ is the smallest value for which $C(r)$ is regular.
Note that $r \ge k$, with equality if and only if $C(k)$ is regular.
If $r>k$, then $C(k+1)$ through $C(r)$ all have regularization level $r$.

A critical germ provides
the basic example of a focal curve germ which is not regularizable.
There are, however, other nonregularizable germs: one can create a curve
germ that agrees with a critical germ to infinite order, i.e., that
cannot be separated from the critical germ by any number of liftings.
In \cite[Theorem 2.36]{MR2599043},
the authors observe that one
way to avoid such curves is to assume that $S$ has an analytic
structure; they show that a noncritical analytic curve is regularizable.
We will take a slightly different approach, hypothesizing ``regularizable''
as needed in our later results.
We will give examples of the definitions in the next section,
which will introduce natural coordinate charts on the monster spaces.

Suppose that
$C(k)$ and $C'(k)$ are regularizable focal curve germs
located, respectively, at points $p$ and $p'$ of $S(k)$.
A local diffeomorphism from $S(k)$ to itself taking
$p$ to $p'$ is said to be an \emph{equivalence
of (regularizable) focal curve germs}
if it takes $C(k)$ to $C'(k)$
and its derivative preserves the focal distribution.
An \emph{invariant of focal curve germs} is a function
on the set of equivalence classes.
Note that an equivalence of focal curve germs
gives us (simply by forgetting the curves)
an equivalence of the Goursat germs at $p$ and $p'$ 
of the focal distribution;
thus every focal curve germ invariant is a Goursat invariant.

We now make some further remarks about 
Goursat and focal curve germ invariants,
without giving full definitions of our terminology;
see the beginning pages of \cite{MR2599043}
for a complete discussion.
By a theorem of B\"acklund \cite{MR1509862},
every equivalence of rank 2 Goursat germs is the prolongation of an equivalence of Goursat germs of rank 2 and corank 1, 
which are contact distributions.  
Thus the pseudogroup of equivalences of Goursat germs
at a point is contained in the pseudogroup of local contactomorphisms. In our Figure~\ref{invdiagram}, the top three boxes are invariants for the action of the smaller pseudogroup, while the bottom three boxes are invariants for the action of the larger pseudogroup.

\section{Coordinates on monster spaces} \label{monstercoordinates}

We now explain how to introduce coordinates on the monster spaces
over a surface $S$, and relate them to the coordinates
previously employed in Lemma~\ref{sandwichequations}.
Let $x$ and $y$ be coordinates on a neighborhood $U$ in $S$.
On $U(k)$ there are $2^k$ charts, each of which is 
a copy of $U \times \AAA^k$,
the product of the base neighborhood $U$ and $k$-dimensional affine space,
and on each chart there are $k+2$ coordinate functions:
the pullback of $x$ and $y$ from $U$, together with
$k$ affine coordinates.
At each level $j$,
by a recursive procedure, 
two of these coordinates are
designated as \emph{active coordinates}.
One is the \emph{new coordinate}
$\nn_j$,
and the other is the \emph{retained coordinate}
$\rr_j$.
In addition, for $j>0$,
a third coordinate is designated as the \emph{deactivated coordinate}
$\dd_j$.

To describe the recursive procedure,
we begin with a chart on $U(j)$ with coordinates
 $\nn_{j}$, $\rr_{j}$, and $\dd_{j}$
together with $j-1$ unnamed coordinates.
At each point of the chart,
the fiber of $\Delta(j)$
(except for the zero vector)
consists of tangent vectors for which
either the restriction of the differential $d\nn_{j}$ or 
that of
$d\rr_{j}$ is nonzero.
Create a chart at the next level by choosing one of the following two options:
\begin{itemize}
\item
Assuming the restriction of $d\rr_{j}$ is nonzero, let
$\nn_{j+1}=d\nn_{j}/d\rr_{j}$; then set
$\rr_{j+1}=\rr_{j}$ and $\dd_{j+1}=\nn_{j}$.
We call this the \emph{ordinary choice}.
\item
Assuming the restriction of $d\nn_{j}$ is nonzero, let
$\nn_{j+1}=d\rr_{j}/d\nn_{j}$; then set
$\rr_{j+1}=\nn_{j}$ and $\dd_{j+1}=\rr_{j}$.
We call this the \emph{inverted choice}.
\end{itemize}
To begin the process we always make an ordinary choice,
but there are two possibilities.
On $U$ the active coordinates are $x$ and $y$,
either of which may be designated as the retained coordinate $\rr_0$;
the other coordinate is $\nn_0$,
and there is no deactivated coordinate.
In every chart the names of the coordinates are $\rr_0$, $\nn_0$, $\nn_1$,
\dots, $\nn_j$, but their meaning depends on the chart.

The charts are given names such as $\cC(oiiooi)$, where each symbol $o$
or $i$ records which choice has been made, either ordinary or inverted.

Alternatively, we name all coordinates using superscripted $x$'s and $y$'s,
as follows. We begin with $x^{(0)}=x$ and $y^{(0)}=y$.
At each level, the two active coordinates will be $x^{(i)}$ and
$y^{(j)}$, for some nonnegative integers $i$ and $j$.
 If we create our chart at the next level by
designating $x^{(i)}$ as the retained coordinate,
then the new active coordinate is $y^{(j+1)}=dy^{(j)}/dx^{(i)}$;
if we designate $y^{(j)}$ as the retained coordinate,
then the new active coordinate is $x^{(i+1)}=dx^{(i)}/dy^{(j)}$.
This notation is meant to suggest the standard usage
in calculus, where the superscript indicates a number of prime marks:
$y^{(1)}=y'$, $y^{(2)}=y''$, etc.
Note, however, that the meaning depends on the choice
of chart, e.g.,
\begin{itemize}
\item
$y'$ means $dy/dx$ if we begin
by retaining $x$,
\item
$y'$ means $dy/dx'$ if we first
retain $y$ and then retain $x'$;
this coordinate is first used at level 2;
\item
$y'$ means $dy/dx''$ if we retain $y$ twice, then retain $x''$, etc.
\end{itemize}

Calculating lifts of parameterized curves is straightforward from the definitions of the coordinates, as we now illustrate.

\begin{example} \label{liftexample}
Beginning with the germ $C$ parameterized by
$x=t^5$ and $y=t^7$, we can calculate the fifth lift $C(5)$ as follows:
\begin{align*}
y' &= dy/dx = \tfrac{7}{5}t^2 \\
x' &= dx/dy' = \tfrac{25}{14}t^3 \\
x'' &= dx'/dy' = \tfrac{375}{196}t \\
y'' &= dy'/dx'' = \tfrac{2744}{1875}t \\
y^{(3)} &= dy''/dx'' = \tfrac{537824}{703125}\,.
\end{align*}
Note that we have chosen the chart $\cC(oioio)$ 
on $S(5)$ in which the germ actually appears; it is located at 
$p=\left(0,0;0,0,0,0,\tfrac{537824}{703125}\right)$.
We call this the associated \emph{curvilinear data point};
the values following the semicolon record the
data of orders 1 through~5. 
The third lift $C(3)$ is nonsingular but not regular;
the regularization level of the beginning germ is 4.

Going in the opposite direction, one can begin with the parametric equations for the two active coordinates
$x''$ and $y^{(3)}$ and integrate to obtain the other equations, supplying the required constant values
(which in this example happen to be zeros).
\end{example}

As the example illustrates, given a point $p \in S(k)$
one can easily find focal curves passing through it:
simply write parametric expressions for the two active
coordinates and then do $k$ appropriate integrations.
This will always be successful, unless one has
chosen a constant parametric expression.
In fact, we can choose these two expressions so that
the resulting focal curve germ $C(k)$ is not critical; 
thus it projects
to a curve germ on $S$ (rather than a single point).
Even more, we can choose the expressions so that
the resulting curve germ $C(k)$ is regular.

\begin{example}
The curve $C(3)$ of Example~\ref{liftexample}
--- parameterized by $x=t^5$, $y=t^7$
and the first three equations of the display ---
is not regular. To find a regular focal curve
passing through $(0,0;0,0,0) \in S(3)$,
one can instead begin with $y'=t$ and $x''=t$;
integration yields the curve on $S$ parameterized
by $x=\tfrac{1}{6}t^3$, $y=\tfrac{1}{8}t^4$.
\end{example}

To obtain, in a selected chart on $S(k)$, 
a set of equations defining the focal bundle $\Delta(k)$
within the tangent bundle $TS(k)$, 
we observe that (\ref{onestep}) tells us that we can repeat the equations defining
$\Delta(k-1)$ within $TS(k-1)$, and we just need to adjoin
a single equation. If, at the last step, one has made the ordinary choice,
then the additional equation is
$d\nn_{k-1}=\nn_{k} \, d\rr_{k-1}$; 
if one has made the inverted choice,
then the additional equation is
$d\rr_{k-1}=\nn_{k} \, d\nn_{k-1}$.
These can be written uniformly as
\begin{equation}\label{adjoined}
d\dd_k=\nn_k \, d\rr_k.
\end{equation}

\begin{example} \label{liftexamplecont}
For the chart $\cC(oioii)$ of Example~\ref{liftexample},
the focal bundle is defined within $TS(5)$ by
\begin{align*}
dy &= y' \, dx \\
dx &=x' \, dy' \\
dx' &=x'' \, dy' \\
dy' &=y'' \, dx'' \\
dy'' &=y^{(3)} \, dx'' \,.
\end{align*}
\end{example}

The coordinates used in the charts of the monster spaces can readily be used to provide the 
local coordinates described in Lemma~\ref{sandwichequations}
(which we subsequently used to explain the code words of Goursat distributions),
as we now explain.

\begin{lemma} \label{comparecoordinates}
Working at a point $p$ in a chart on the monster space $S(k)$ with $k\ge 2$,
and assuming $1 \le j \le k-1$,
consider the bundles shown here:
\[
\begin{array}{ccccc} 
&& \Delta(k)_j & \!\!\!\!\subset\!\!\!\! & \phantom{D}\Delta(k)_{j+1}  \\[10pt]
&&  \cup & \ \ &   \ \ \\[10pt]
\cL(\Delta(k)_j) & \subset & \cL(\Delta(k)_{j+1}) &  &\ \
\end{array}
\]
The equations of these bundles inside $\Delta(k)_{j+1}$
are as follows:
\begin{equation} \label{eqboxesmonster}
\begin{array}{cccccc} 
& & \boxed{d\dd_{k-j+1}= \nn_{k-j+1} d\rr_{k-j+1}}  & \!\!\!\!\subset\!\!\!\! & \boxed{\phantom{XXX}} \\[10pt]
 & \ \ & \cup  \ \ \\
\boxed{\begin{array}{c} d\rr_{k-j+1} = 0 \\ d\dd_{k-j+1}= 0 \\ d\nn_{k-j+1} = 0 \end{array}} &  \subset \!\!\!\!\!\!\!\! & \boxed{\begin{array}{c} \phantom{.} d\rr_{k-j+1} = 0 \phantom{.} \\ \phantom{.} d\dd_{k-j+1}= 0 \phantom{.} \end{array}} &\ \
\end{array}
\end{equation}

Thus we may obtain appropriate Kumpera--Ruiz coordinates for Lemma~\ref{sandwichequations} by letting
$x=\rr_{k-j+1}-\rr_{k-j+1}(p)$, $y=\dd_{k-j+1}-\dd_{k-j+1}(p)$, and $y'=\nn_{k-j+1}-\nn_{k-j+1}(p)$,
and letting the constant of part~\ref{p1} of that lemma be $C=\nn_{k-j+1}(p)$.
\end{lemma}

\begin{proof}
Equation~(\ref{moresteps})
tells us that, for $2 \le j \le k-1$,
the sheaves $\Delta(k)_j$ are extensions from lower levels.
Thus it suffices to verify that the equations of \ref{eqboxesmonster}
are correct when $j=1$, i.e., that the equations are as follows:
\begin{equation} \label{topeqboxesmonster}
\begin{array}{cccccc} 
& & \boxed{d\dd_{k}= \nn_{k} d\rr_{k}}  & \!\!\!\!\subset\!\!\!\! & \boxed{\phantom{XXX}} \\[10pt]
 & \ \ & \cup  \ \ \\
\boxed{\begin{array}{c} d\rr_{k} = 0 \\ d\dd_{k}= 0 \\ d\nn_{k} = 0 \end{array}} & \!\!\!\!\subset\!\!\!\! & \boxed{\begin{array}{c} d\rr_{k} = 0 \\ d\dd_{k}= 0 \end{array}} &\ \
\end{array}
\end{equation}
This is done by induction on $k$.

For $k=2$
we are considering the bundle $\Delta(2)$.
If the ordinary choice has been made,
then the identifications of variables are as follows:
\begin{align*}
\rr_2 &= x \\
\dd_2 &= y' \\
\nn_2 &= y''.
\end{align*}
If the inverted choice has been made,
then the identifications are
\begin{align*}
\rr_2 &= y' \\
\dd_2 &= x \\
\nn_2 &= x'.
\end{align*}

For the inductive step, 
we claim that we have the following diagram
of equations inside $\Delta(k)_{3}$ for the bundles 
$\Delta(k)$, $\Delta(k)_{2}$, and $\Delta(k)_{3}$
(in the top row)
and the Cauchy characteristics $\cL(\Delta(k)_{2})$ and $\cL(\Delta(k)_{3})$
(in the bottom row).
\begin{equation*}
\begin{array}{cccccc} 
\boxed{\begin{array}{c} d\dd_{k-1}= \nn_{k-1} d\rr_{k-1} \\ d\dd_{k}= \nn_{k} d\rr_{k} \end{array}} & \subset & \boxed{d\dd_{k-1}= \nn_{k-1} d\rr_{k-1}}  & \!\!\!\!\subset\!\!\!\! & \boxed{\phantom{XXX}} \\[10pt]
\cup & \ \ & \cup  \ \ \\
\boxed{\begin{array}{c} d\rr_{k-1} = 0 \\ d\dd_{k-1}= 0 \\ d\nn_{k-1} = 0 \end{array}} & \!\!\!\!\subset\!\!\!\! & \boxed{\begin{array}{c} d\rr_{k-1} = 0 \\ d\dd_{k-1}= 0 \end{array}} &\ \
\end{array}
\end{equation*}
Except for the box at top left, this
follows from the inductive hypothesis;
in that box we have adjoined
the additional equation (\ref{adjoined}) defining $\Delta(k)$ inside $\Delta(k)_2$.
Now erase the rightmost boxes of each row.
For the remaining three boxes,
we obtain equations inside $\Delta(k)_{2}$
by omitting the equation 
$d\dd_{k-1}= \nn_{k-1} d\rr_{k-1}$
from the boxes of the top row,
and using this equation to simplify the set of equations in the bottom left box,
i.e., by omitting $d\dd_{k-1}=0$.
There are now two cases.
If the ordinary choice has been made, then
we replace $\rr_{k-1}$ by $\rr_k$ and $\nn_{k-1}$ by $\dd_k$;
if the inverted choice has been made, then
we replace $\rr_{k-1}$ by $\dd_k$ and $\nn_{k-1}$ by $\rr_k$.
In either case, we obtain the right box of the bottom row of
(\ref{topeqboxesmonster}).
The equations in the left box of this row are a consequence of
part~\ref{ldequations} of Lemma~\ref{sandwichequations}.

The last statement of the lemma now follows from comparing 
diagrams (\ref{eqboxes}) and (\ref{eqboxesmonster}).

\end{proof}

\begin{example} \label{mscgc}
Using the chart $\cC(oioio)$ on $S(5)$
of Examples~\ref{liftexample} and
\ref{liftexamplecont}, here is diagram \ref{topeqboxesmonster}:
\begin{equation*} 
\begin{array}{cccccc} 
& & \boxed{dy''= y^{(3)} dx''}  & \!\!\!\!\subset\!\!\!\! & \boxed{\phantom{XXX}} \\[10pt]
 & \ \ & \cup  \ \ \\
\boxed{\begin{array}{c} dx'' = 0 \\ dy''= 0 \\ dy^{(3)} = 0 \end{array}} & \!\!\!\!\subset\!\!\!\! & \boxed{\begin{array}{c} dx'' = 0 \\ dy''= 0 \end{array}} &\ \
\end{array}
\end{equation*}
The bundle $\cL(\Delta(5))$ at lower left is the trivial (rank 0) bundle.
For the point $p$
located at
$$
\left(x(p),y(p);y'(p),x'(p),x''(p),y''(p),y^{(3)}(p)\right),
$$
we obtain local coordinates by subtracting each of these constants
from the corresponding coordinate: $X:=x-x(p)$, $Y:=y-y(p)$,
$Y':=y'-y'(p)$, etc.
Using these coordinates, $\cL(\Delta(5)_2)$
is  defined  inside $\Delta(5)_2$
by 
\[
dX''=dY''=0,
\]
while the focal bundle $\Delta(5)$ is defined inside $\Delta(5)_2$ 
by the single equation
\[
dY'' = \left(y^{(3)}(p) + Y^{(3)}\right)dX''.  
\]
Using this analysis to infer the last symbol in the code word associated with the focal bundle $\Delta$ at point $p$ (following the recipe of Section~\ref{RVTGoursat}), we note that it depends on the location of $p$.
If $y^{(3)} = 0$, the last symbol is $T$, and otherwise it is $R$.
\end{example}

\section{Divisors at infinity and their prolongations} \label{strata}

As explained in Section~\ref{MT},
one constructs the monster tower over a surface
by prolongation of successive focal bundles:
the monster
space
$S(j)$ is obtained by applying the general prolongation
construction
of Section~\ref{PMC}
to the bundle $\Delta(j-1)$ on  $S(j-1)$. For $j\ge 2$, the rank 2 bundle $\Delta(j-1)$ 
contains the line bundle $T(S(j-1)/S(j-2))$ of vertical vectors.
If one applies the general prolongation construction 
to 
this line bundle, one obtains a tower of spaces
\begin{equation*} 
\cdots \to I_j[3] \to I_j[2] \to I_j[1]  \to I_j  \to S(j-1)
\end{equation*}
over $S(j-1)$;
since we started with a line bundle, each of these
spaces is simply a copy of $S(j-1)$.
Each of these spaces naturally fits inside
its counterpart in the monster tower.
Thus $I_j$ is a divisor on $S(j)$, while $I_{j}[1]$
is a submanifold of codimension 2 on $S(j+1)$, etc.
We call $I_j$ the $j$th \emph{divisor at infinity}
and $I_j[\ell]$ its $\ell$th \emph{prolongation}.
(For a generalization, see \cite{MR3720327}.)

What is the geometric meaning of these loci?
Using either our interpretation of the coordinate systems,
or reflecting upon how the lift of a curve could possibly
have a vertical tangent, we realize that the divisor at infinity $I_j$
consists of those curvilinear data points for which we consider the 
$j\text{th}$-order data to be infinite.

\begin{example}
For the \emph{ramphoid cusp}
of Figure~\ref{ramphoid},
its lift to $S(2)$ is given by
\begin{align*}
x&=t^2 \\
y&=t^4+t^5 \\
y' &=2t^2+\tfrac{5}{2}t^3 \\
y'' &=2+\tfrac{15}{4}t .
\end{align*}
Observe that this lift is
tangent to the fiber of $S(2)$ over $S(1)$.
We can compute
$$
y'''\ =\frac{dy''}{dx}=\frac{15}{8t}
$$
but of course we can't evaluate this at $t=0$;
that's because the third lift intersects $I_3$.
To see the intersection we instead should compute
$$x' =\frac{dx}{dy''}=\frac{8}{15}\,t\,. $$
In this chart the divisor at infinity $I_3$
is $x'=0$.
\end{example}

\begin{figure}[htbp]
   \centering
   \includegraphics[scale=0.75]{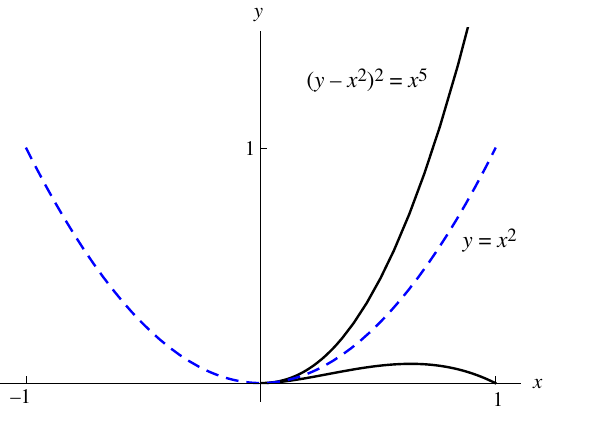} 
   \caption{A ramphoid cusp.}
   \label{ramphoid}
\end{figure}

In general one obtains the divisor at infinity $I_j$ by making
the inverted choice in going from level $j-1$ to level $j$.
In the resulting chart, $I_j$ is given by the vanishing
of the new coordinate $\nn_j$.
To obtain the prolongations $I_j[\ell]$,
one continues by making $\ell$ ordinary choices,
and $I_j[\ell]$ is given by the additional vanishing of 
$\nn_{j+1}$ through $\nn_{j+\ell}$.

To compare all these divisors at infinity and their prolongations,
let's pull everything back to a chosen level $k$.
The complete inverse image of $I_j$ will be denoted in the same way,
but now it's a divisor up on $S(k)$. With this convention,
we now have a nest
$$
I_j \supset I_j[1] \supset I_j[2] \supset \cdots \supset I_j[k-j].
$$

\begin{lemma} \label{compareloci}
Working at a point $p\in S(k)$, where $k\ge 2$, consider the focal  bundle
$\Delta(k)$. Let $3 \le j \le k$.
The loci $N(V)_j$ and $N(VT^{\tau})_j$
from (\ref{eq: vertical locus}) and (\ref{eq: VTtau locus}) are
the divisor at infinity $I_j$
and the prolongation
 $I_{j-\tau}[\tau]$.
\end{lemma}
\begin{proof}
This is immediate from the identification of Goursat and monster coordinates
given in Lemma~\ref{comparecoordinates}.
\end{proof}

\section{Code words redux} \label{redux}

In Section~\ref{RVTGoursat} we have associated code words with Goursat distributions. We now associate similar code words with points on the monster spaces over a surface $S$, and to focal curve germs on these spaces. We also discuss how these code words are related to each other.

We define an \emph{RVT code word} to be a finite word in these three symbols, satisfying the following rules:
\begin{enumerate}
\item
The first symbol is $R$.
\item \label{trule}
The symbol $T$ may only be used immediately following a $V$ or $T$.
\end{enumerate}
For the Goursat words of Section~\ref{RVTGoursat}, by contrast, the first two symbols are required to be $RR$.
Given an RVT code word $W$, we associate with it a Goursat word 
$G(W)$ by this simple procedure: if there is a $V$ in second position, then replace it and any immediately succeeding $T$'s by $R$'s.
For example, if $W$ is $RVTTTVRVT$ then $G(W)$ is $RRRRRVRVT$.

\subsection{Code words of points}

Given a point $p \in S(k)$, the $k$th monster space over a surface $S$,
we associate with it an RVT code word of length $k$ by considering where it lies
with respect to the divisors at infinity and their prolongations.
Since the divisors at infinity arise beginning at level $2$,
the first symbol in the code word is an automatic $R$.
Here are the rules for assigning the symbol in position $j$:
\begin{enumerate}
\item \label{Vrule}
If $p$ lies on the $j$th divisor at infinity $I_j$,
then the $j$th symbol is $V$.
\item \label{Trule}
If $p$ lies on the prolongation $I_h[\ell]$,
where $\ell>0$ and $h+\ell=j$,
then the $j$th symbol is $T$.
\item
Otherwise the $j$th symbol is $R$.
\end{enumerate}

Rules~\ref{Vrule} and \ref{Trule} appear to clash,
but in fact it's impossible for $p$ to lie on both
$I_j$ and $I_h[\ell]$.
To see this, we consider assigning the symbols
one step at a time, beginning after the automatic $R$;
we look at the sequence of points going up the tower
to our specified point on $S(k)$.
At each step, we are considering a fiber of
$S(j+1)$ over a point $p \in S(j)$, a projective line.
Let $q$ be a point on the fiber of $S(j+1)$ over $p$.
One point on the fiber is the intersection with the
divisor at infinity $I_{j+1}$;
if $q$ is this point, then the new symbol is $V$.
If $p$ lies on $I_j$ or some $I_h[\ell]$,
there is a second special point
on the fiber, the intersection with the prolongation $I_j[1]$ or $I_h[\ell+1]$;
if $q$ is this point, then the new symbol is $T$.
These two special points are distinct.
This description also makes it clear that the code word associated with a point
satisfies rule~\ref{trule} for RVT code words.

\begin{example}
For the point $p$ of Example~\ref{liftexample}
the associated RVT code word is $RVTVR$;
for the origin of the chart $\cC(oioio)$ used there, the RVT code word is $RVTVT$.
\end{example}
%

\subsection{Code words of focal curve germs} \label{cwofcg}

To obtain the RVT code word of a regularizable focal curve germ
$C(k)$ on $S(k)$,
lift it upward through the 
monster tower until you reach the regularization level $r$.
The RVT code word records 
how the successive lifts meet the divisors at infinity
$I_{k+2}$ through $I_r$,
with a $V$ indicating that
we are on the divisor at infinity, a $T$ indicating that we are
on the prolongation of a divisor at infinity, and $R$ being used otherwise.
Also use an $R$ in the first position, i.e.,
ignore the divisor at infinity $I_{k+1}$
and its prolongations,
as well as those of the prior divisors at infinity $I_{2}$ through $I_k$.
To explain the process in a slightly different two-step way:
\begin{enumerate}
\item
Record how the curve and its
lifts meet the divisors at infinity $I_{k+1}$ through $I_r$
and their prolongations.
\item
If the resulting word begins with $VT^\tau$,
replace this string by $R^{\tau+1}$.
\end{enumerate}
The first step of this process produces a word ending with a critical
symbol $V$ or $T$, but this symbol may be changed to $R$ in the second step.
As a special case, we associate a code word with a curve $C$ on
a surface $S$. Here there is no divisor at infinity on $S(1)$,
so that the initial $R$ is automatic (unless the germ is nonsingular,
in which case the word is empty),
and the final symbol will definitely be critical.

\begin{example}
For the curve germ $C$ of Example~\ref{liftexample}
(whose regularization level is 4),
the associated RVT code word is $RVTV$;
for $C(1)$ through $C(4)$, each of which also has regularization level 4,
the associated code words are
$RRV$, $RV$, $R$, and the empty word.
The fifth lift $C(5)$ has regularization level 5 and its code word is likewise empty.
\end{example}

The reason we make this construction for focal curve germs,
rather than just curves on surfaces, is that it gives
us the flexibility to calculate other invariants recursively.
For example, it is used in the subsequent Section~\ref{PC}
to develop a front-end recursion for Puiseux characteristics.

By definition, all of our code words thus far are finite.
For later usage, however, it is convenient to associate
a code word of infinite length with a focal curve germ,
simply by padding out with an infinite string of $R$'s.
None of the invariants considered in this paper is altered
when one appends $R$'s to the end of a code word;
thus it makes sense to speak of the invariants
associated with such infinite words. For example,
$PC(RVVRRR\dots)=PC(RVV)=[3;5]$.

One could assign a code word to a nonregularizable focal curve germ,
but such a word would be of infinite length and would contain
an infinite string of $T$'s; we will not develop a theory of these words.

\subsection{Words of points vs.\ words of Goursat germs}\label{wpwg}

Recall that in Section~\ref{RVTGoursat} we associated a Goursat word with any Goursat distribution at a point.

\begin{theorem}
Working at a point $p\in S(k)$, consider the focal bundle $\Delta(k)$.
The Goursat word associated with $\Delta(k)$ at $p$ 
is $G(W)$, where $W$ is the RVT code word associated with $p$.
\end{theorem}

\begin{proof}
We compare the rules for computing the Goursat word
of $\Delta(k)$ at $p$ with
the rules for computing $G(W)$, where $W$ is the RVT code
word associated with $p$.
Because of the recursive nature of both sets of rules, it suffices to
show that they give the same result in the final position.

Lemma~\ref{compareloci} and (\ref{eq: vertical locus}) tell us that, for $k\ge3$,
\[
D_{k}(p)=\cL(D_{k+2})(p) \iff p \in I_k.
\]
Thus
the rules for whether to use a $V$ in the final position of
the code word agree.
Now suppose $k\ge \tau+3$, where $\tau\ge 1$, and that
$\pi(p)\in N(VT^{\tau-1})_k$,
where $\pi:S(k)\to S(k-1)$ is projection.
Then, again using Lemma~\ref{compareloci} and (\ref{eq: VTtau locus}),
\begin{align*}
\text{the Goursat word of} &\;\Delta(k)\; \text{at $p$ has final symbol $T$} \\
& \iff p \in N(VT^{\tau})_k \\
& \iff p \in I_{k-\tau}[\tau] \\
& \iff \text{the RVT code word of $p$ has final symbol $T$.}
\end{align*}
In all other situations, both sets of rules dictate a final symbol $R$.
\end{proof}

As we have explained in Section~\ref{universality},
the process of finding a point
$p_k \in S(k)$ representing a given Goursat distribution
can be interpreted as finding a sequence
of points $p_1\in S(1)$, $p_2\in S(2)$, etc.
Since all points of $S(2)$ represent the same Goursat distribution,
we may choose $p_2$ wherever we like.
In particular we may choose a point $p_2$ not lying on the divisor
at infinity $I_2$; in this case the RVT code word is already a Goursat word.

\subsection{Words of points vs.\ words of focal curves}

Suppose that $C(k)$ is a regularizable focal curve germ located at a point
$p\in S(k)$.
The RVT code word of $p$ and the RVT code word of $C(k)$
record complementary information.
Let $C$ be the curve germ on $S$
obtained by projecting $C(k)$.
Then the RVT code word of $p$
records the sequence of divisors at infinity
and prolongations of such divisors one encounters
when lifting $C$ through the tower to recover $C(k)$;
it is a word of length $k$.
The RVT code word of the germ $C(k)$ records the divisors
at infinity and prolongations that one encounters
as one continues to lift $C(k)$ until one reaches
the regularization level $r$;
it is a word of length $r-k$.

\begin{example} \label{141819}
Consider this focal curve germ $C(2)$ located at the origin of chart $\cC(oi)$:
\begin{align*}
x &= t^{14} \\
y &= 14(t^{18}+t^{19}) \\
y' &= t^{4}(18+19t)\\
x' &= \frac{14t^{10}}{72+95t}
\end{align*}
The code word of $p=(0,0;0,0)$ is $RV$.
The regularization level is 7, and the code word of
of $C(2)$ is $\text{\emph{RRVRV}}$.
Compare these code words with the code word of $C$:
\[
RV|\text{\emph{TTVRV}}
\]
We have inserted a vertical slash to indicate how one can
recover the words of $p$ and $C(2)$; to obtain a valid
RVT code word, one needs to replace the initial two $T$'s by $R$'s.
\end{example}

For a curve germ $C$ on $S$, its code word is the same
as the word of the location of $C(r)$ on $S(r)$. More
generally, as the example illustrates, one can find the
code word of $C(k)$ on $S(k)$ by finding the code word
of the location of $C(r)$, lopping off the first $k$ symbols,
and then replacing $V$'s and $T$'s as necessary to
obtain a valid RVT code word.

\begin{example}
In Example~\ref{141819}, the location of $C(7)$
in chart $\cC(oiooioi)$ is
\[
(0,0;0,0,0,0,0,\tfrac{8707129344}{1225},0)
\]
and the code word of this point is $\text{\emph{RVTTVRV}}$.
\end{example}

Going in the other direction, given a point
$p \in S(k)$ we can find a regular curve germ $C$ in $S$
whose $k$th lift passes through $p$;
one simply specifies appropriate parameterizations of the two
active coordinates and obtains the other parametric equations
by integration.
The code word of $C$ is then the same as the code word of $p$.

\subsection{The lifted word}

Suppose that $W$ is the RVT code word associated with a 
regularizable focal curve germ $C(k)$
on $S(k)$.
The code word of its lift $C(k+1)$ is the
\emph{lifted word} $L(W)$ obtained from $W$ by this recipe:
\begin{itemize}
\item
Remove the first symbol $R$.
\item
If the new first symbol is $V$, replace it by $R$,
and likewise replace any immediately succeeding $T$'s
by $R$'s. Stop when you reach the next $R$, the next $V$, or the end of the word.
\end{itemize}
Observe that this lifting procedure also applies to an infinite word.

\begin{example}
If $W$ is $\text{\emph{RVTTVRVRV}}$,
then $L(W)$ is $\text{\emph{RRRVRVRV}}$.
\end{example}

Going in the opposite direction, assume that we know $L(W)$
and want to know the possibilities for $W$.
There is a choice in the first step:
\begin{enumerate}
\item
Any string $R^{\tau}$ at the beginning of $L(W)$ can be replaced
by $VT^{{\tau}-1}$. The choice ${\tau}=0$ is allowed; this means don't
make a replacement at all.
\item
Put the symbol $R$ at the beginning.
\end{enumerate}
\begin{example} \label{4words}
If $L(W)$ is $\text{\emph{RRRVRVRV}}$, then there are four possibilities for $W$:
\[
\begin{gathered}
\text{\emph{RRRRVRVRV}} \\
\text{\emph{RVRRVRVRV}} \\
\text{\emph{RVTRVRVRV}} \\
\text{\emph{RVTTVRVRV}} \\
\end{gathered}
\]
\end{example}

\section{Comparison with embedded resolution} \label{blowup}

As noted in the Introduction, singularity theorists working in the context of 
algebraic geometry tend to work with point blowups of varieties rather than Nash modifications.
The ideas and calculations of earlier sections have counterparts in this context,
which we now briefly explain.
For a related treatment, see
\cite{MR2434453}, especially Sections 9 and 14.

If we begin, as in Section \ref{monstercoordinates}, with coordinates
$x= x_0$ and $y = y_0$ 
on a neigh\-bor\-hood $U$ on the surface $S$ which contains point $p$, then we may define \emph{blowup coordinates} on the $k$th blowup
$S\{k\}$ of $S$ in a recursive manner: the new coordinate is either 
\[
y_{j+1} = (y_{j}-y_{j}(p_{k-1}))/(x_{i}-x_{i}(p_{k-1}))
\]
or
\[
x_{i+1} = (x_{i}-x_{i}(p_{k-1}))/(y_{j}-y_{j}(p_{k-1})).
\]
where $p_{k-1}$ denotes a point of $S\{k-1\}$ in the fiber over $p$.
The sequence of points and the choice of which of the two formulas
to use is dictated by the curve on $S$ which one is studying.
Note the contrast with the monster construction,
which is carried out in a uniform manner, no matter what
curve is eventually to be studied.

\begin{example} \label{blowupexample}
Working with the curve $C$ parameterized by $x = t^5$, $y = t^7$, the blowup coordinate calculation 
over the origin analogous to that of Example~\ref{liftexample}  is
\begin{align*}
y_{1} &= y/x =t^2 \\
x_{1} &= x/y_{1} = t^3 \\
x_{2} &= x_{1}/y_{1} = t \\
y_{2} &= y_{1}/x_{2} = t \\
y_{3} &= y_{2}/x_{2}  = 1\,.
\end{align*}
From this we see that the strict transform $C_3 \subset S\{3\}$ is nonsingular.
\end{example}

In the context of embedded resolution, the symbols of the RVT code 
of Subsection~\ref{cwofcg}
have the following meanings:
\begin{itemize}
\item
The first symbol is automatically $R$.
\item
The symbol in position $j$ is $V$ if the strict transform $C_j$
meets the strict transform of the exceptional divisor $E_{j-1}$.
\item
The symbol in position $j$ is $T$ if $C_j$
meets the strict transform of some earlier exceptional divisor $E_{i}$
(with $i<j-1$).
This will necessarily be the same exceptional divisor that it met at the
prior level.
\item
The symbol in position $j$ is $R$ if $C_j$ meets only the newest
exceptional divisor $E_j$.
\end{itemize}

We say that the strict transform $C_j$ is \emph{regular}
if it is nonsingular and if it transverse to all the exceptional
divisors that it meets. The \emph{regularity level}
is the lowest level at which these conditions are met;
this definition agrees with that of Section~\ref{fcg}.

\section{Puiseux characteristics} \label{PC}

\subsection{The Puiseux characteristic of a focal curve germ}
\label{Puiseuxchar}

The Puiseux characteristic is a well-known 
invariant of a curve germ on a surface $S$,
but we wish to extend it to a broader context
and relate it to the RVT code.
Let $p$ be a point on $S(k)$
and let $C(k)$ be a regularizable focal curve germ
located at $p$.
Choose an ordered pair $(X,Y)$ of coordinate functions vanishing at $p$,
such that
\begin{itemize}
\item
the differentials $dX$ and $dY$ are independent linear functionals on the focal plane at $p$, and
\item
$X$ has the smallest possible order of vanishing.
\end{itemize}
Denoting this order by $n$,
introduce a parameter $t$ for which $X=t^n$.
Then $Y$ may be expressed as a formal power series
\begin{equation} \label{yseries}
Y = \sum a_i t^i
\end{equation}
where we write just those terms for which $a_i\ne 0$.
(To say this another way, we write $Y$ as a fractional power series
$Y = \sum a_i X^{i/n}$.)
An exponent in this series is called \emph{essential}
or \emph{characteristic}
if it is not divisible by the greatest common divisor
of the smaller exponents.
In this definition we count $n$ as an essential exponent;
thus the first essential exponent appearing in (\ref{yseries})
is the first exponent (if any) not divisible by $n$.
The set of essential exponents 
$\{n, \lambda_1, \dots, \lambda_g\}$
is finite and their greatest common divisor is 1.
Letting $\lambda_0=n$,
we define the \emph{Puiseux characteristic} to be
\[
[\lambda_0;\lambda_1, \dots, \lambda_g].
\]

For example, for the focal curve $C$ parameterized by
\begin{align*}
x &= s^4 \\
y &= s^6 + s^7 \\
y' &= \tfrac{3}{2} s^2 + \tfrac{7}{4} s^3
\end{align*}
we may take our first coordinate function to be $y'$
and thus use parameter 
$t=s \sqrt{\frac{3}{2} + \frac{7}{4} s}$.
For second coordinate we may use $x$, whose
expansion in the new parameter is
\[
x = \tfrac{4}{9} t^4 - \tfrac{28\sqrt{6}}{81} \, t^5 + \cdots
\]
the first term being inessential. The Puiseux characteristic
is $[2;5]$.
Observe that $y$ may not be used as our second coordinate 
function, since $dy = y' dx=0$ on the focal plane at $p$.

\subsection{The front-end recursion for Puiseux characteristic} \label{frontendPC}

In what follows,
we show that the RVT code word
associated with a regularizable focal curve determines its
Puiseux characteristic.
As we will see, the 
Puiseux characteristic can be computed by either of two
recursions:
\begin{enumerate}
\item
A front-end recursion whose step size is a single symbol of the word.
\item
A back-end recursion whose step size is a block of symbols.
\end{enumerate}
Here we state the front-end
recursion and prove
that it
is valid, i.e., that it correctly computes
the Puiseux characteristic of each focal curve.
Subsequently we will
describe the back-end recursion
and prove that it is a consequence of the
front-end recursion;
thus it is likewise valid.

The front-end recursion uses
the notion of the lifted word
$L(W)$
of an RVT code word $W$,
as defined in Section~\ref{redux}.

\begin{theorem} \label{ABCrules}
Suppose that $C(k)$ is a regularizable focal curve germ
located at a point $p$ on $S(k)$.
Its associated code word $W$
determines its
Puiseux characteristic.
If $W$ has
no critical symbols, then
the associated Puiseux characteristic, denoted
$\PC(W)$, is $[1;]$.
In general the Puiseux characteristic
is determined recursively by
the following rules.
Suppose that
\[
\PC(W) = [\lambda_0;\lambda_1,\dots,\lambda_g].
\]

\begin{enumerate}[label=(\Alph*)]
\item  \label{TypeA}
The word $W$ begins with $RR$ if and only if
$\lambda_1>2\lambda_0$. In this case, $\PC(L(W))$ is obtained from $\PC(W)$
by keeping $\lambda_0$ and subtracting it from all other entries:
\[
\PC(L(W))= [\lambda_0;\lambda_1-\lambda_0,\dots,\lambda_g-\lambda_0].
\]
\item \label{TypeB}
The word $W$ begins with $RVT^{\tau}R$
or is $W=RVT^{\tau}$ if and only if
\[
\lambda_0=({\tau}+2)(\lambda_1-\lambda_0)
\quad \text{and} \quad
\lambda_1=({\tau}+3)(\lambda_1-\lambda_0). 
\]
In this case
\[
\PC(L(W))= [\lambda_1-\lambda_0;\lambda_2-(\lambda_1-\lambda_0),
\dots,\lambda_g-(\lambda_1-\lambda_0)].
\]

(Note that this Puiseux characteristic is shorter.)
\item \label{TypeC}
The word $W$ begins with $RVT^{\tau}V$
if and only if
\[
({\tau}+1)(\lambda_1-\lambda_0) < \lambda_0 < ({\tau}+2)(\lambda_1-\lambda_0).
\]
In this case
\[
\PC(L(W))= [\lambda_1-\lambda_0;\lambda_0,\lambda_2-(\lambda_1-\lambda_0),
\dots,\lambda_g-(\lambda_1-\lambda_0)].
\]
\end{enumerate}
\end{theorem}

%
%

\begin{example}
The Puiseux characteristics associated with the four words of Example~\ref{4words}
are
\begin{align*}
\PC(\text{\emph{RRRRVRVRV}}) &= [8;36,38,39] \\
\PC(\text{\emph{RVRRVRVRV}}) &= [16;24,36,38,39] \\
\PC(\text{\emph{RVTRVRVRV}}) &= [24;32,36,38,39] \\
\PC(\text{\emph{RVTTVRVRV}}) &=  [28;36,38,39] \\
\end{align*}
The first word is covered by Theorem \ref{ABCrules}\ref{TypeA}, the next two by
\ref{TypeB}, and the last word by \ref{TypeC}.
In each case $L(W)$ is $\text{\emph{RRRVRVRV}}$, so that $\PC(L(W))=[8;28,30,31]$.
\end{example}

The formulas of Theorem~\ref{ABCrules} already appear in 
Theorem 3.5.5 of Wall's textbook \cite{MR2107253},
but we remark that the circumstances are different in several ways:
Wall's formulas apply to a curve on a surface,
a special case of a focal curve;
they are obtained using blowups rather than lifts
(i.e., Nash modifications);
and he doesn't invoke RVT code words.
Nevertheless our proof is similar.

\begin{proof}
Let $X$ and $Y$ be chosen in accordance with our definition of Puiseux characteristic:
\begin{align*}
X&=t^{\lambda_0} \\
Y &= \sum a_i t^i.
\end{align*}
We may choose $Y$ so that the first term in its series is $a_{\lambda_1} t^{\lambda_1}$.
For the lifted curve we have available a new coordinate function .
\[
Y'=dy/dx=\sum \frac{i}{\lambda_0}a_i t^{i-\lambda_0}.
\]

If $\lambda_1>2 \lambda_0$, then $X$ continues to have the lowest order of vanishing;
thus $(X,Y')$ is an appropriate pair of coordinate functions for computing the Puiseux
characteristic of the lifted curve, and we obtain the formula of \ref{TypeA}.
The next stage in lifting involves $Y''=dY/dX$, so that the code word begins with $RR$.

If $\lambda_1<2 \lambda_0$, then $\lambda_1-\lambda_0< \lambda_0$ 
and
we can use $(Y',X)$ as an appropriate ordered 
pair of coordinate
functions.
Let $[\widetilde{\lambda}_0;\widetilde{\lambda}_1,\ldots]$ be the Puiseux characteristic of the lifted curve;
we observe that 
$\widetilde{\lambda}_0=\lambda_1-\lambda_0$.
To compute the other characteristic exponents we need to invert the fractional
power series for $Y'$.
The formulas for this inversion are well known:
we apply 
\cite[Proposition 5.6.1]{MR1782072},
replacing Casas-Alvero's
$[n;m_1,\dots]$
by
$[\lambda_1-\lambda_0;\widetilde{\lambda}_1,\dots]$
and his
$[\overline{n};\overline{m}_1,\dots]$
by
$[\lambda_0;\lambda_1-\lambda_0,\dots]$.
The proposition tells us that
there are two mutually exclusive possibilities:
\begin{itemize}
\item
The leading characteristic exponent
$\lambda_1-\lambda_0$ divides
$\lambda_0$,
and the other characteristic exponents
are determined by
\[
(\lambda_{i+1}-\lambda_0)+\lambda_0=\widetilde{\lambda}_i
+(\lambda_1-\lambda_0).
\]
\item
The first two characteristic exponents are
$\lambda_1-\lambda_0$ and 
$\widetilde{\lambda}_1=\lambda_0$;
the remaining characteristic exponents
are determined by
\[
(\lambda_{i}-\lambda_0)+\lambda_0=\widetilde{\lambda}_i
+(\lambda_1-\lambda_0).
\]
\end{itemize}
These two possibilities lead to the formulas
for $\PC(L(W))$
of \ref{TypeB} and \ref{TypeC}.

In order to compute the associated code word, we examine the 
successive derivatives $X',X'',X^{(3)},\dots$
of $X$ with respect to $Y'$.
The order of vanishing in $t$ of $X^{(i)}$ is
$(i+1)\lambda_0-i\lambda_1$.
In case  \ref{TypeB}, let ${\tau}$ be the nonnegative integer for which
$\lambda_0=({\tau}+2)(\lambda_1-\lambda_0)$.
In this case the order of vanishing of $X^{(i)}$ is 
$({\tau}+2-i)(\lambda_1-\lambda_0)$,
which is positive when $i<{\tau}+2$ and becomes
zero when $i={\tau}+2$.
Thus when we evaluate at $t=0$ we obtain 0
for the first ${\tau}+1$ derivatives, whereas $X^{({\tau}+2)}$
has a nonzero value.
Thus the code word begins with $RVT^{{\tau}}$;
if this is not the entire word then the next symbol
is $R$.
Similarly, in case \ref{TypeC} we let ${\tau}$
be such that
\[
({\tau}+1)(\lambda_1-\lambda_0) < \lambda_0 < ({\tau}+2)(\lambda_1-\lambda_0),
\]
and observe that the first ${\tau}+1$ derivatives evaluate
to 0. The next derivative $X^{({\tau}+2)}$ becomes
infinite at $t=0$, which means that the associated symbol
is $V$. Thus the code word begins with $RVT^{\tau}V$.
\end{proof}

Here we restate the recipes
of Theorem~\ref{ABCrules}
in a manner better suited for straightforward calculation.

\begin{theorem} \label{PCfront}
Suppose that the Puiseux characteristic of $L(W)$ is
\[
\PC(L(W)) = [\lambda_0;\lambda_1,\dots,\lambda_g].
\]

\begin{enumerate}[label=(\Alph*)]
\item
If $W$ begins with $RR$, then
$\PC(W)= [\lambda_0;\lambda_1+\lambda_0,\dots,\lambda_g+\lambda_0]$.
\item
If $W$ begins with $RVT^{\tau}R$
or if $W=RVT^{\tau}$,
then
\[
\PC(W)= [({\tau}+2)\lambda_0;({\tau}+3)\lambda_0,\lambda_1+\lambda_0,\dots,\lambda_g+\lambda_0].
\]
(This is a longer Puiseux characteristic.)
\item
If $W$ begins with $RVT^{\tau}V$, then
$\PC(W)= [\lambda_1;\lambda_1+\lambda_0,\dots,\lambda_g+\lambda_0]$.
\end{enumerate}
\end{theorem} 

\begin{example} \label{frontendexample}

Let
$W = RRVTRRRVTTTV$.
To compute the Puiseux characteristic using the front-end recursion, we start with 
$$ \PC(R)=[1;\ ].
$$
Theorem \ref{PCfront}(B) with $\tau=0$ gives 
$$ \PC(RV)=[2;3 ].
$$
Then using Theorem  \ref{PCfront}(A) three times gives 
$$ \PC(RRRRV)=[2;9 ].
$$
Theorem  \ref{PCfront}(C) gives 
$$ \PC(RVTTTV)=[9;11 ].
$$
Using Theorem  \ref{PCfront}(A) four times gives 
$$ \PC(RRRRRVTTTV)=[9;47 ].
$$
Then Theorem  \ref{PCfront}(B) with $\tau=1$ gives 
$$ \PC(RVTRRRVTTTV)=[27;36,56 ].
$$
Finally, Theorem  \ref{PCfront}(A)  gives 
$$ \PC(RRVTRRRVTTTV)=[27;63,83 ].
$$

\end{example}

\subsection{The back-end recursion for Puiseux characteristic} \label{backendPC}
 
We now turn to the back-end recursion.
This is similar to what is found in 
\cite[Section 3.8]{MR2599043} and \cite{MR3152113},
but our circumstances are slightly broader and
the details of our recursion are slightly different.
 A contiguous substring in an RVT code word is called
 \emph{critical} if the last symbol is $V$ or $T$;
we call it
\emph{entirely critical} if each of its symbols is $V$ or $T$ (i.e., if it
contains no $R$'s).
Appending an $R$ to the end of a code word doesn't change
the associated Puiseux characteristic, and for a word consisting
entirely of $R$'s (including the empty word) the associated Puiseux characteristic is $[1;]$.
Thus the back-end recursion deals with critical words.

 We first define two auxiliary functions $E_T$ and $E_V$
 acting on ordered pairs of positive integers:
 \begin{align*}
 E_T[a;b]&=[a;a+b] \\
  E_V[a;b]&=[b;a+b].
 \end{align*}
 We recursively define a function $E$ from entirely critical strings
 to ordered pairs of positive integers:
\begin{equation*}
\begin{aligned}
E(\text{empty string}) &=[1;2] \\
E(TQ) &= E_T(E(Q)) \\
E(VQ) &= E_V(E(Q)).
\end{aligned}
\end{equation*}
(In these formulas $Q$ denotes an entirely critical string, possibly empty.)
In fact, we can extend it to strings of the form $R^{{\rho}}Q$, where $Q$
is entirely critical, by defining $E_R[a;b]=[a;a+b]$
(so that $E_R$ is the same as $E_T$) and
$$
E(RQ) = E_R(E(Q)).
$$
For example,
$$
E(R^4 V^2)=E_R\circ E_R\circ E_R\circ E_R\circ E_V\circ E_V[1;2]=[3;17].
$$

\begin{theorem} \label{PCback}
The Puiseux characteristic of a critical RVT code word is determined as follows.
Write the code word $W$ as
$PR^{{\rho}}Q$, where 
$P$ is either empty or critical,
${\rho} \geq 1$,
and
$Q$ is entirely critical. 
(Note that the substring $Q$ must begin with a $V$,
since we want $W$ to be a valid code word.)
Suppose that
\begin{equation*}
\begin{aligned}
\PC(P) &= [\lambda_0;\lambda_1,\ldots,\lambda_g] 
\quad \text{and}\\
E(R^{{\rho}}Q) &= [a;b].
\end{aligned}
\end{equation*}
Then
\[
\PC(W) 
= [a \lambda_0; a \lambda_1,\ldots, a \lambda_g, a \lambda_g +b-2a].
\]
\end{theorem}

If $W=R^{{\rho}}Q$, i.e., if $P$ is the empty word,
then (applying our rules with $g=0$) we have
\begin{equation*}
\PC(P)=[1;] \quad \text{and} \quad
\PC(W)=[a;b-a].
\end{equation*}
In other words, the back-end recursion tells us that
\begin{equation} \label{11A}
\PC(R^{\rho}Q)=E(R^{{\rho}-1}Q).
\end{equation}
Within this recursion,
the subroutine of computing the function $E$ is a front-end recursion.
In fact it's essentially a special case of the front-end recursion
of Theorem  \ref{PCfront}. Indeed,
we can prove formula (\ref{11A})
by induction on the length of the word.
To begin, we remark that for a word with a single $V$ 
we have agreement:
\begin{align*}
\PC(R^{\rho}VT^{\tau})&=[\tau+2;\rho(\tau+2)+1] \\
E(R^{\rho-1}VT^{\tau})&=[\tau+2;\rho(\tau+2)+1].
\end{align*}
Thus in the inductive step it suffices to consider a word
$W=R^{\rho}VT^{\tau}Q'$, where $Q'$ begins
with a $V$.
Let $\PC(L(W))=[\lambda_0; \lambda_1]$. There are two cases to consider:
\begin{itemize}
\item
If $\rho \ge 2$, then by the inductive hypothesis
\[
\PC(R^{\rho-1}VT^{\tau}Q')=E(R^{\rho-2}VT^{\tau}Q')=[\lambda_0;\lambda_1].
\]
According to Theorem \ref{PCfront}(A), 
\[
\PC(R^{\rho}VT^{\tau}Q')=[\lambda_0;\lambda_1+\lambda_0],
\]
in agreement with
\[
E(R^{\rho-1}VT^{\tau}Q')
=E_R(E(R^{\rho-2}VT^{\tau}Q'))=[\lambda_0;\lambda_1+\lambda_0].
\]
\item
If $\rho=1$, then by the inductive hypothesis
\[
\PC(R^{\tau+1}Q')=E(R^{\tau}Q')=[\lambda_0;\lambda_1].
\]
According to Theorem \ref{PCfront}(C), 
\[
\PC(RVT^{\tau}Q')=[\lambda_1;\lambda_1+\lambda_0],
\]
in agreement with
\[
E(VT^{\tau}Q')=E_V(T^{\tau}Q'))=E_V(R^{\tau}Q')=[\lambda_1;\lambda_1+\lambda_0].
\]
\end{itemize}

\begin{example}

As in Example \ref{frontendexample}, let $W = RRVTRRRVTTTV$.
To compute its Puiseux characteristic using the back-end recursion, we start by 
using formula (\ref{11A}):
$$ \PC(RRVT)=E(RVT)=[3;7].
$$
Now compute $E(RRRVTTTV)=[9;38]$, so that Theorem \ref{PCback} gives
$$ \PC(RRVTRRRVTTTV)=[9\cdot 3; 9\cdot 7, 9\cdot 7 + 38 -2\cdot 9]=[27;63,83].
$$
\end{example}

\begin{proof}[Proof of Theorem \ref{PCback}]
To prove the formula for $\PC(W)$, we 
use induction on the length of $W$.
For a word with a single entirely critical block,
we have already confirmed formula (\ref{11A}).
Thus in the inductive step it suffices to consider a word
$W=PR^{{\rho}}Q$, where 
$P$ is critical,
${\rho} \geq 1$,
and
$Q$ is entirely critical.
We remark that the lifted word is
\[
L(W) = L(P) R^{{\rho}}Q.
\]
Let $\PC(L(P))=[\lambda_0;\lambda_1,\ldots,\lambda_g]$
and $E(R^{{\rho}}Q) = [a;b]$.
By the inductive hypothesis we know that
\[
\PC(L(W))
=[a \lambda_0; a \lambda_1,\ldots, a \lambda_g, a \lambda_g +b-2a].
\]
We now consider the
three cases of Theorem \ref{PCfront}:

\begin{enumerate}[label=(\Alph*)]
\item
If $P$ begins with $RR$, then \ref{PCfront}(A) tells us that
\begin{align*}
\PC(P)&= [\lambda_0;\lambda_1+\lambda_0,\dots,\lambda_g+\lambda_0] \\
\PC(W)&=  [a \lambda_0;a \lambda_1+ a \lambda_0,\dots,
a \lambda_g+a \lambda_0,a \lambda_g +b-2a+a \lambda_0] \\
&=[a \lambda_0;a (\lambda_1+ \lambda_0),\dots,
a (\lambda_g+ \lambda_0),a (\lambda_g+ \lambda_0) +b-2a].
\end{align*}

\item
If $P$ begins with $RVT^{\tau}R$
or if $P=RVT^{\tau}$,
then \ref{PCfront}(B) tells us that
\begin{align*}
\PC(P)&= [({\tau}+2)\lambda_0;({\tau}+3)\lambda_0,\lambda_1+\lambda_0,\dots,\lambda_g+\lambda_0] \\
\PC(W)
&=  [({\tau}+2)a\lambda_0;({\tau}+3)a\lambda_0, \\
&\qquad a\lambda_1+a\lambda_0,\dots,a\lambda_g+a\lambda_0,
a\lambda_g+b-2a+a\lambda_0] \\
&=  [a({\tau}+2)\lambda_0;a({\tau}+3)\lambda_0, \\
&\qquad a(\lambda_1+\lambda_0),\dots,a(\lambda_g+\lambda_0),
a(\lambda_g+\lambda_0)+b-2a] .
\end{align*}

\item
If $P$ begins with $RVT^{\tau}V$, then \ref{PCfront}(C) tells us that
\begin{align*}
\PC(P)&= [\lambda_1;\lambda_1+\lambda_0,\dots,\lambda_g+\lambda_0] \\
\PC(W)&= [a \lambda_1;a \lambda_1+a \lambda_0,\dots,a \lambda_g+a \lambda_0,
a \lambda_g+b-2a+a \lambda_0] \\
&=[a \lambda_1;a(\lambda_1+\lambda_0),\dots,a(\lambda_g+\lambda_0),
a(\lambda_g+\lambda_0)+b-2a].
\end{align*}

\end{enumerate}

\end{proof}

\subsection{From Puiseux characteristic to code word} \label{our}

The Puiseux characteristic and the RVT code word are equivalent
pieces of information. Having obtained two recursive recipes for computing
the Puiseux characteristic from a given code word, we now explain how to
reverse these recipes.

In the case of front-end recursion, the reversed recipe can be 
inferred from the cases of Theorem \ref{ABCrules}.
We illustrate the process by an example.
\begin{example}
Suppose that $\PC(W) = [10;15, 27]$. This puts us in case \ref{TypeB} of Theorem \ref{ABCrules} with ${\tau} = 0$; thus $W$ begins with $\text{\emph{RVR}}$ and $\PC(L(W)) = [5; 22]$.
Invoking case \ref{TypeA} three times, we learn that $L(W)$, $L^2(W)$, and $L^3(W)$ all begin with $RR$, so that
\[    W = \text{\emph{RVRRR}}\dots    \]
and $\PC(L^4(W)) = [5; 7]$. Using case \ref{TypeC} with ${\tau} = 1$, we find that $L^4(W)$ begins with $RVTV$ and that $\PC(L^5(W)) = [2; 5]$. By case \ref{TypeA}, $L^5(W)$ begins with $RR$ (but this is redundant information) and $\PC(L^6(W)) = [2; 3]$. Finally we use case \ref{TypeB} with ${\tau} = 0$ to infer that $L^6(W)$ begins with $\text{\emph{RVR}}$. Thus 
\[
W = \text{\emph{RVRRRVTVR}}.
\]

\end{example}

For the back-end recursion, the reverse recipe is essentially what
is explained in \cite[Section 3.8.4]{MR2599043}
and in \cite[Section 2.4]{MR3152113},
although there the circumstances are slightly different. 
We denote the map from Puiseux characteristics to critical RVT code words by $\CW$.
We first describe an auxiliary map which we denote by $\Euc$;
given a pair $(a,b)$ of relatively prime positive integers with $a<b$,
it produces a string of $V$'s and $T$'s.
We define $\Euc(1,2)$ to be the empty word.
Otherwise we recursively define
$$
\Euc(a,b)=
\begin{cases}
V \text{ followed by } \Euc(b-a,a) \quad \text{ if } b<2a,\\
T  \text{ followed by } \Euc(a,b-a) \quad \text{ if } b>2a.
\end{cases}
$$
For example,
$$
\Euc(2,7)=T\Euc(2,5)=TT\Euc(2,3)=TTV.
$$
And here is an example using Fibonacci numbers: $\Euc(F_{k+2},F_{k+3})=V^k$.

We now describe the $\CW$ map. First suppose that the Puiseux characteristic is simply 
$[\lambda_0;\lambda_1]$. To obtain $\CW[\lambda_0;\lambda_1]$,
first compute $\Euc(\lambda_0,\lambda_1)$, then
replace the initial string of $T$'s (if any) by the same-length string of $R$'s
and put a single $R$ in front.
For example, $\CW[2;7]=RRRV$.
For a longer Puiseux characteristic $[\lambda_0;\lambda_1,\dots,\lambda_{k+1}]$, let
\begin{equation*}
\begin{aligned}
a &= \GCD(\lambda_0,\lambda_1,\dots,\lambda_{k}),\\
s &= \left\lfloor \frac{\lambda_{k+1}-\lambda_{k}}{a} \right\rfloor + 1, \\
b &= a \left( \frack \left( \frac{\lambda_{k+1}-\lambda_{k}}{a} \right) + 1 \right)
\end{aligned}
\end{equation*}
(using the floor function and the fractional-part function as indicated).
Then inductively define
$$
\CW[\lambda_0;\lambda_1,\dots,\lambda_{k+1}]
=\CW\left[\frac{\lambda_0}{a};\frac{\lambda_1}{a},\dots,\frac{\lambda_k}{a}\right]
R^s \Euc(a,b).
$$
For example,
\begin{equation*}\CW[4;6,7]
=\CW[2;3]R\Euc(2,3)=\text{\emph{RVRV}}
\end{equation*}
(since $a=2$, $s=1$, and $b=3$).

\subsection{The Puiseux characteristics of points and Goursat germs}

Given a point $p_k \in S(k)$, we associate with it a well-defined Puiseux
characteristic, namely the Puiseux characteristic associated with
any regular focal curve germ passing through $p_k$.
Similarly, given a Goursat germ of corank $k$,
recall from Section~\ref{wpwg} that we can choose a representative point
on $S(k)$ that avoids the second divisor at infinity $I_2$. All such points have the
same associated Puiseux characteristic. Thus we may speak of the Puiseux characteristic
of a Goursat germ. This is a \emph{restricted Puiseux characteristic},
meaning that it satisfies the inequality $\lambda_1>2 \lambda_0$.
If instead we had chosen a point lying on $I_2$, to obtain the restricted
Puiseux characteristic we would need to replace $\lambda_0$
by the remainder obtained by dividing $\lambda_1$ by $\lambda_0$.

The version in \cite{MR2599043} and \cite{MR3152113} of the algorithm for obtaining the Puiseux characteristic 
is applied to code words in which the first symbol is at level 3. In comparison
with our algorithm, here's how it works: first put $RR$ at the beginning, 
thus obtaining a Goursat word,
and then
apply our algorithm. Note that this yields a restricted Puiseux characteristic. 
For the algorithm in the opposite direction, 
begin with a restricted Puiseux characteristic.
If we apply our $\CW$ map,
we obtain a code word beginning with $RR$; the algorithm in \cite{MR2599043} and \cite{MR3152113} jettisons
these initial symbols.


\section{The multiplicity sequence} \label{multseq}

One of the standard invariants in the theory of curves on surfaces is the
\emph{multiplicity sequence}, as explained, e.g., in Section 3.5 of  \cite{MR2107253}.
The usual definition employs embedded resolution of singularities,
whereas for present purposes we want to use lifting into the monster tower;
the two definitions give identical sequences.
We also want to generalize the situation, by defining the
multiplicity sequence associated with any regularizable focal curve germ.

Consider a regularizable focal curve germ $C(k)$ located at point $p_k$ on the monster space $S(k)$.
For $j\ge k$,
the symbol
$m_j$ denotes the multiplicity of $C(j)$ at its location $p_{j}$ on $S(j)$.
The \emph{multiplicity sequence}
of $C(k)$
is the sequence $m_k,m_{k+1},m_{k+2},\dots$.
Note that
$m_j$ is the leading entry of the Puiseux characteristic $\PC(L^{j-k}(W))$,
where $W$ is the associated infinite code word of Section~\ref{cwofcg}.
One can regard each multiplicity as being associated
with a symbol of the word, except for the leading multiplicity $m_k$,
as indicated in Figure~\ref{proxdiagram}.
(The edges will be explained momentarily.)
This infinite code word ends in an infinite string of $R$'s,
for which each of the associated multiplicities is 1.

Our recursion for computing the Puiseux characteristics of lifted curves
yields a recursion for computing the multiplicity sequence from the code word.

\begin{example}

For the word $W=\text{\emph{RVTVVRRR}}\dots$,
the Puiseux characteristics are as follows:
\begin{align*}
\PC(RVTVV) &= [8;11] \\
\PC(RRVV) &= [3;8] \\
\PC(RVV) &= [3;5] \\
\PC(RV) &= [2;3] \\
\PC(R) &= [1; ] \\
\PC( ) &= [1; ] 
\end{align*}
Thus the multiplicity sequence is $8,3,3,2,1,1,1,\dots$.
\end{example}

An alternative recursion uses the proximity diagram.
Since our situation is different
from that of Wall \cite{MR2107253},
we cannot simply cite his definitions,
but eventually we will obtain a formula identical to that
of his Proposition 3.5.1(iii).
Here are our definitions.
Looking at the sequence of points $p_k$, $p_{k+1}$, $p_{k+2}$,
etc.,
we say that $p_j$ is \emph{proximate} to $p_i$ if
either $j=i+1$, or
$p_j$ lies on the prolongation of the divisor at infinity $I_{i+2}$,
where $i\ge k$.
The  \emph{proximity diagram} represents each point of the sequence by a vertex, with edges recording proximity; we furthermore record the code word, as in Figure~\ref{proxdiagram}
(taken from page 52 of \cite{MR2107253} and modified).
The diagram continues to the right indefinitely, but the essential
information is visible in that portion to the left of the last critical symbol.

\begin{figure}[!htb]
\begin{center}
\includegraphics[scale=0.8]{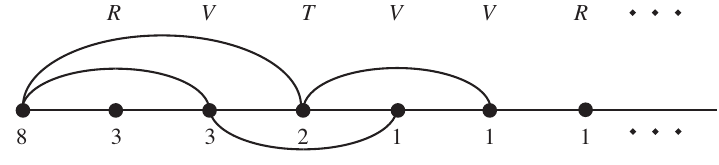}
\end{center}
\caption{A proximity diagram.}
\label{proxdiagram}
\end{figure}

We observe that the code word determines the diagram.
The edges between successive vertices are automatic.
The other edges are determined in the following way:
for each substring $VT^\tau$, let $i$ be the position 
of the vertex corresponding to the $V$,
and draw edges from each vertex of the substring to the vertex 
in position $i-2$.
The multiplicities to the right of the last critical symbol are 1's.

\begin{theorem} \label{wallmultformula}

The multiplicity $m_i$ is given by
$\sum m_j$,
summing over all $j$ for which $p_j$ is proximate to $p_i$.
\end{theorem}

\begin{proof} We prove this by induction on the length of the code word $W$.
In the inductive step we use the three
cases of Theorems~\ref{ABCrules} and \ref{PCfront}.
Let the Puiseux characteristic of $L(W)$ be
$[\lambda_0;\lambda_1,\dots,\lambda_g]$.
In case (A) the word begins with $RR$; thus 
in the proximity diagram the leftmost vertex is connected
only to its neighbor to the right.  Theorem~\ref{PCfront}(A)
tells us that the multiplicities of these two vertices
are both $\lambda_0$.
In case (B), the leftmost portion of the proximity diagram is as shown in
Figure~\ref{pdb}.
\begin{figure}[!htb]
\begin{center}
\includegraphics[scale=0.7]{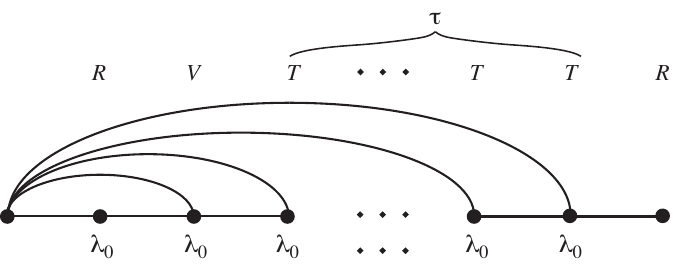}
\end{center}
\caption{Left portion of the proximity diagram in case (B).
There may be additional edges going rightward from the rightmost
vertex labeled $T$, but they are irrelevant to the argument.}
\label{pdb}
\end{figure}
\begin{figure}[!htb]
\begin{center}
\includegraphics[scale=0.7]{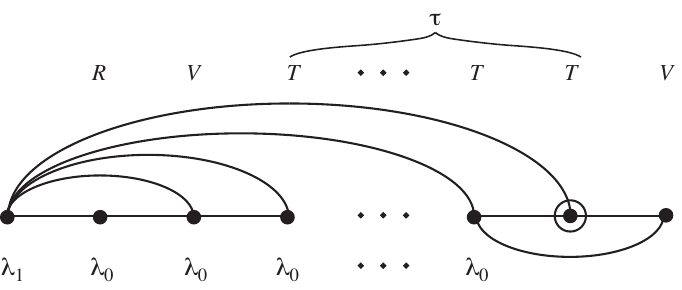}
\end{center}
\caption{Left portion of the proximity diagram in case (C).
There may be additional edges going rightward from the two rightmost
vertices labeled $T$, but again they are irrelevant.}
\label{pdc}
\end{figure}
There the inductive hypothesis has been used to infer
that all the indicated multiplicities are the same.
Theorem~\ref{PCfront}(B) tells us that the multiplicity
at the left vertex is $(\tau+2)\lambda_0$,
in accord with the formula of Theorem~\ref{wallmultformula}.
In case (C), Theorem~\ref{PCfront}(C)
tells us that the leftmost multiplicity is $\lambda_1$.
Thus the inductive hypothesis tells us that
the leftmost portion of the proximity diagram
is as shown in Figure~\ref{pdc}.
By Theorem~\ref{ABCrules}(A) we infer that
\[
PC(L^i(W)) = PC(R^{\tau}V\cdots)=[\lambda_0,\lambda_1-(i-1)\lambda_0,\dots]
\]
for $1\le i \le \tau+1$;
then by Theorem~\ref{ABCrules}(B) or (C)
we find that the first entry of 
$PC(L^{\tau+2}(W))$
is $(\lambda_1-\tau\lambda_0)-\lambda_0
=\lambda_1-(\tau+1)\lambda_0$.
Thus this is the multiplicity associated with the circled vertex.
Again this is in accord with 
the formula of Theorem~\ref{wallmultformula}.
 \end{proof}

\section{Vertical orders} \label{vert}

For a regularizable focal curve germ $C(k)$ located at $p_k\in S_k$,
let $C(r)$ be its lift to the regularization level $r$.
For $k+2 \le j \le r$,
we define the \emph{vertical order} $\VO_j$ to be the intersection
number at $p_r$ of $C(r)$ and the divisor at infinity $I_j$:
\[
\VO_j := (C(r) \cdot I_j)_{p_r}
\]
Recall that the divisor at infinity $I_j$ first appears at level $j$
of the tower. To interpret the intersection number in our definition,
one can either use its inverse image on $S(r)$
--- it is convenient to continue denoting it by $I_j$ ---
or one can project $C(r)$ to a curve germ on $S(j)$
(likely a singular curve)
and compute the intersection number there.
By the standard projection formula (for example, see \cite[p.\ 25]{MR0735435} or \cite[Example 8.1.7]{MR1644323}),
these procedures give the same result.
As a practical matter one computes the vertical order by using a parameterization
of $C(r)$, and $\VO_j$ is the order of vanishing of
the function defining the divisor at infinity $I_j$.

The  \emph{vertical orders vector} is
\[\VO(C(k))=(\VO_{k+2},\VO_{k+3},\dots,\VO_r)\]
and
the  \emph{restricted vertical orders vector} is
\[\RO(C(k))=(\VO_{k+3},\VO_{k+4},\dots,\VO_r).\]
We will soon show that these vectors just depend on the word
$W$ associated with $C(r)$; thus the notations $\VO(W)$
and $\RO(W)$ may be used.
In fact $\RO(W)$ is determined by the Goursat word $G(W)$
and can thus be interpreted as an invariant of Goursat distributions;
this is why we introduce this invariant, which will appear again in our subsequent paper.

\begin{example} \label{VOexample}
Consider
a plane curve with code word $W=RVVVRVT$,
equivalently Puiseux characteristic $[15;24,25]$.
One such curve is $x=t^{15}, y=t^{24}+t^{25}$;
alternatively one can obtain such a curve by starting
in the chart $\cC(oiiioio)$
with the parameterizations of the active variables
$x^{(3)}=1+t$ and $y^{(4)}=t$
and then integrating appropriately.
One deduces the following orders of vanishing:
\begin{align*}
\ord_t(x) &= 15 \\
\ord_t(y) &= 24 \\
\ord_t(y') &= 9 \\
\VO_2=\ord_t(x') &= 6 \\
\VO_3=\ord_t(y'') &= 3 \\
\VO_4=\ord_t(x'') &= 3 \\
\ord_t(x^{(3)}) &= 0 \\
\VO_6=\ord_t(y^{(3)}) &= 2 \\
\ord_t(y^{(4)}) &= 1 
\end{align*}
The fifth and seventh lifts of the curve
do not meet the divisors at infinity $I_5$ and $I_7$,
so that the corresponding vertical orders are zero.
Thus $\VO(W)=(6,3,3,0,2,0)$.

For a curve with code word $G(W)=\text{\emph{RRVVRVT}}$,
equivalently Puiseux characteristic $[9;24,25]$,
we find that 
$\VO(G(W))=(0,3,3,0,2,0)$.
\end{example}

When one works with ordinary blowups rather than Nash modifications,
the vertical order has the following interpretation: 
$\VO_j$ is the intersection number of $C_j$ and $E_{j-1}$,
where $C_j$ is the strict transform of the plane curve $C$
and $E_{j-1}$ is the strict transform of the exceptional
divisor arising from the $(j-1)$st blowup.
Again as a practical matter, one computes orders of vanishing,
as illustrated by the following example.

\begin{example}
Working as in Example~\ref{VOexample}
with the curve $x=t^{15}, y=t^{24}+t^{25}$,
the orders of vanishing presented there are the same as the
orders of vanishing of 
$x$, $y$,
$y_1=y/x$,
$x_1=x/y_1$,
$y_2=x_1/y_1$, 
$x_2=x_1/y_2$,
$x_3=x_2/y_2$,
$y_3=y_2/(x_3-1)$,
and
$y_4=y_3/(x_3-1)$.
Here we have used the blowup coordinates explained
in Section~\ref{blowup}.
\end{example}

\begin{theorem}
For $0\le j \le r-2$, we have $\VO_{j+2}=m_j - m_{j+1}$.
Thus the vertical orders depend only on the word associated
with a focal curve germ.
\end{theorem}
\begin{proof}
Once again we refer to the three cases of 
Theorem~\ref{ABCrules},
letting $W$ be the word associated with $C_j$.
In case (A), the multiplicity $m_j$ is the leading entry of
$PC(W)$ and $m_{j+1}$ is the leading entry of
$PC(L(W))$;
both of them are $\lambda_0$.
Since $\lambda_1>2 \lambda_0$,
$C_{j+2}$ does not meet the divisor at infinity $I_{j+2}$
and thus $VO_{j+2}=0$.
In cases (B) and (C),
we have $m_j=\lambda_0$
and $m_{j+1}=\lambda_1-\lambda_0$.
To see the germ $C_{j+1}$, we need to make the inverted choice,
and the order of vanishing of the new coordinate
(whose vanishing gives the divisor at infinity)
is $\lambda_0-(\lambda_1-\lambda_0)$.
\end{proof}


\section{The small growth sequence} \label{sgs}

Suppose that $D$ is a 
distribution 
on a manifold $M$; let $\cE$ be its sheaf of sections.
We consider its
\emph{small growth sequence}:
$$
\cE = \cE^1 \subset \cE^2 \subset \cE^3 \subset \dots
$$
defined by
$$
\cE^j = [\cE^{j-1},\cE],
$$
meaning the subsheaf of $\Theta_M$ whose sections are generated by
Lie brackets of sections of $\cE^{j-1}$
with sections of $\cE$ and by the sections of $\cE^{j-1}$.
Note how this differs from the Lie square sequence in (\ref{Liesquaresheaves}): at each step we form
Lie brackets with vector fields from the beginning distribution.
For each point $p \in X$ we let
$\SG_i$ denote the rank of $\cE^i$ at $p$;
we call
\[
\SG(D,p) = \SG_1,\SG_2,\ldots
\]
the \emph{small growth vector}.
Simple examples show that this vector may differ from point to point of $M$.

In our subsequent paper, we will consider the small growth vector of a Goursat distribution and related invariants
of the small growth sequence, as indicated in the bottom box of Figure~\ref{invdiagram}, as well as their connections to the invariants considered in this paper.


\bibliography{mybib}{}
\bibliographystyle{plain}


\end{document}